%% file: CS5_revisited_2019August21.tex
\documentclass[12pt]{amsart}

\usepackage{wasysym}
\usepackage{amsmath}
\usepackage{amssymb}
\usepackage{amsthm}
\usepackage{tikz}
\usetikzlibrary{matrix,arrows,backgrounds,shapes.misc,shapes.geometric,patterns,calc,positioning}
\usetikzlibrary{calc,shapes}
\usepackage{wrapfig}
\usepackage{epsfig}  
\usepackage{color}	 %
\input{xy}
\xyoption{poly}
\xyoption{2cell}
\xyoption{all}

\usepackage{soul} 

\usepackage{lscape}

\usetikzlibrary{calc,shapes}
\usetikzlibrary{matrix,arrows,backgrounds,shapes.misc,shapes.geometric,patterns,calc,positioning}
\usetikzlibrary{calc,shapes}

\newcommand{\caln}{\mathcal{N}}

\newcommand{\calg}{\mathcal{G}}
\newcommand{\calh}{\mathcal{H}}

\newcommand{\calgSW}{{}_{SW}\calg}
\newcommand{\calgNE}{\calg^{N\!E}}

\newtheorem{thm}{Theorem}[section]
\newtheorem{prop}[thm]{Proposition}

\newtheorem{cor}[thm]{Corollary}
\newtheorem{conjecture}[thm]{Conjecture}

\theoremstyle{defn}
\newtheorem{definition}[thm]{Definition}

\theoremstyle{remark}

\newtheorem{example}[thm]{Example}
\newtheorem{remark}[thm]{Remark}

\newtheorem{thmIntro}{\bf{Theorem}}

\newtheorem{corIntro}{\bf{Corollary}}

\numberwithin{equation}{section}

\newcommand{\cfq}{[a_1,a_2,\ldots,a_n]}
\newcommand{\cfa}{[a_1,a_2,\ldots,a_n]}
\newcommand{\caf}{[a_n\ldots,a_2,a_1]}
\newcommand{\cfb}{[b_1,b_2,\ldots,b_n]}

\newcommand{\cfg}{\calg [a_1,a_2,\ldots,a_n]}

  \def\sgn{\operatorname{sgn}}

\newcommand{\ze}{\epsilon}
\newcommand{\zg}{\gamma}

\newcommand{\band}{\calg^\circ}

\begin{document}
\title{Snake graphs and continued fractions}
\author{\.{I}lke  \c{C}anak\c{c}{\i}}
\address{School of Mathematics, Statistics and Physics, 
Newcastle University,
Newcastle Upon Tyne NE1 7RU, 
United Kingdom}
\email{{ilke.canakci@ncl.ac.uk}}
\author{Ralf Schiffler}\thanks{The first author was supported by EPSRC grants EP/N005457/1 and EP/P016014/1 and the Durham University, and the second author was supported by the NSF grants  DMS-1254567 and DMS-1800860,  and by the University of Connecticut.}
\address{Department of Mathematics, University of Connecticut, 
Storrs, CT 06269-1009, USA}
\email{schiffler@math.uconn.edu}




%
%
%
\begin{abstract}
This paper is a sequel to our previous work in which we found a combinatorial realization of continued fractions as quotients of the number of perfect matchings of snake graphs. We show how this realization reflects the convergents of the continued fractions as well as the Euclidean division algorithm. We apply our findings to establish results  on sums of squares, palindromic continued fractions, Markov numbers and other statements in elementary number theory.
\end{abstract}

 \maketitle


\tableofcontents


\section{Introduction}
Snake graphs are planar graphs that appeared first in the theory of cluster algebras. Cluster algebras are subalgebras of a field of rational functions generated by {\em cluster variables} \cite{FZ}.  
A special type of cluster algebras are those associated to marked surfaces, see \cite{FST,FT}, which have been studied by many people, see for example \cite{BZh,FG1,GSV, FeShTu,QZ}.
 For these cluster algebras it was shown in \cite{MS,MSW} that for every cluster variable 
there is a snake graph 
 such that the cluster variable is given as a sum over all perfect matchings of the snake graph,  
where each term in this sum is a Laurent monomial in  two types of variables $x_1,x_2,\ldots,x_N$ and $y_1,y_2,\ldots,y_N$.
In \cite{MSW2}, this formula was used to construct canonical bases for the cluster algebra using snake graphs and also band graphs, which are obtained from snake graphs by identifying two edges. One special case was later provided in \cite{CLS}. These results were generalized to orbifolds in \cite{CT,FeTu}. 

In our previous work, \cite{CS,CS2,CS3}, we studied snake graphs from an abstract point of view, and constructed a ring of snake graphs and band graphs which reflects the  relations between the elements of  cluster algebras of surface type in terms of bijections between sets of perfect matchings of unions of snake and band graphs. 

In our most recent paper \cite{CS4}, we established a bijection between continued fractions $\cfa$ and snake graphs $\calg\cfa$, such that the number of perfect matchings of the snake graph equals the numerator of the continued fraction. Moreover, we showed that this equation of natural numbers can be lifted to the cluster algebra by expressing the cluster variables as Laurent polynomials in $x_1,x_2,\ldots,x_N$. In this formula the $y$-variables were set to 1, meaning that the cluster algebra has trivial coefficients. In \cite{R}, this formula was generalized to include the $y$-variables.

Thus we have a formula that writes a cluster variable $u$, which is a Laurent polynomial in variables  $x_1,x_2,\ldots,x_N$,$y_1,y_2,\ldots,y_N$, as the numerator  of a continued fraction of Laurent polynomials $L_1,L_2,\ldots,L_n$ such that, when we specialize all variables $x_i=y_i=1$, we obtain 
\[m(\calg)= \textup{ numerator of }\cfa,\]
where $m(\calg)$ is the number of perfect matchings of the snake graph $\calg=\calg\cfa$ and $a_i$ is the specialization of $L_i$.

A different specialization has been studied in \cite{LS6}, setting $x_i=1, y_2=y_3=\cdots=y_N=-t^{-1}$, and $y_1=t^{-2}$. Curiously, this specialization computes the Jones polynomial of the 2-bridge link associated to the continued fraction.

\smallskip

In this paper, we concentrate on the specialization $x_i=y_i=1$. We apply our results from \cite{CS}--\cite{CS4} to establish several statements in elementary number theory.
First, we note how certain  automorphisms of the snake graph translate to the continued fractions.

\begin{thmIntro}[Theorem~\ref{thm pal1}]
 \label{thm pal1}
 A snake graph  has a rotational symmetry at its center tile if and only if the corresponding continued fraction is palindromic of even length.
\end{thmIntro}

Given a snake graph $\calg=\calg\cfa$, we introduce its palindromification as the snake graph $\calg_\leftrightarrow$ associated to $[a_n,\dots,a_2,a_1,a_1,a_2,\dots,a_n]$. Then we show the following.

\begin{thmIntro} [Theorem~\ref{thm palsnake}]
 Let $\calg=\calg[a_1,a_2,\ldots,a_n]$ be a snake graph  and $\calg_\leftrightarrow$ its palindromification. Let $\calg'=\calg[a_2,\ldots,a_n]$.
 Then 
\[m(\calg_{\leftrightarrow}) = m(\calg)^2 + m(\calg')^2.\]
\end{thmIntro}

As a consequence we obtain the following corollary.
\begin{corIntro} [Corollary~\ref{cor squares}]\begin{itemize}
\item[{\rm(a)}] 
 If $M$ is  a sum of two relatively prime squares then there exists a palindromic continued fraction of even length whose numerator is $M$.
\item[{\rm(b)}] For each positive integer $M$, the number of ways $M$ can be written as a sum of two relatively prime squares is exactly one half of  the number of palindromic even length continued fractions  with numerator $M$.
\end{itemize}
\end{corIntro}

We also study palindromic snake graphs of odd length and give a formula whose numerator is a difference of two squares; however this formula is not reduced.

We then apply our results to Markov numbers. By definition, these are the integer solutions to the Markov equation $x^2+y^2+z^2=3xyz$. It was shown in \cite{BBH,Propp} that Markov numbers correspond to the cluster variables of the cluster algebra of the torus with one puncture. The snake graphs of these cluster variables are therefore called {\em Markov snake graphs.}
Each Markov number, hence each Markov snake graph, is determined by a line segment from $(0,0)$ to a point $(q,p)$ with a pair of relatively prime integers and $0<p<q$. 
We give a simple realization of the Markov snake graph using the Christoffel path from the origin to the point $(q,p)$, Section~\ref{sect: Christoffel}. 

 Let us mention that infinite continued fractions were used in \cite{Ser1} to express the Lagrange number $\sqrt{9m^2-4}/m$ of a Markov number $m$. Our approach here is different, since we express the Markov number itself as the numerator of a finite continued fraction. See also \cite{Ser2} for an explicit connection between geodesics on the modular surface (the quotient of the hyperbolic plane by the modular group) and continued fractions.

Finally, we study Markov band graphs which are constructed from the Markov snake graphs by adding 3 tiles and then identifying two edges. Geometrically, the Markov band graphs correspond  to the closed simple curves obtained by  moving the arc of the Markov snake graph infinitesimally away from the puncture. 
\begin{thmIntro} [Theorem~\ref{thm band}] Let $m$ be a Markov number and let $\band(m)$ be its Markov band graph.
 Then the number of perfect matchings of $\band(m)$ is  $3m$.
\end{thmIntro}
This result has an interesting connection to number theory, because if $(m_1,m_2,m_3)$ is a solution of the Markov equation $x^2+y^2+z^2=3xyz$ then $(3m_1,3m_2,3m_3)$ is  a solution of the equation $x^2+y^2+z^2=xyz$.

Let us point out some other combinatorial approaches  to continued fractions.  In \cite{BQS} the authors gave a combinatorial interpretation of continued fractions as the number of ways to tile a strip of length n with dominoes of length two and stackable squares of length one. This was used in  \cite{BZ} to prove that every prime of the form $4m+1$ is the sum of two relatively prime squares. In \cite{AB}, certain palindromic conditions on the coefficients of an infinite continued fraction $[a_1,a_2, ....]$ were used to deduce transcendence of the corresponding real number.  The authors also considered weaker quasi-palindromic conditions of which our `almost palindromes' of section 3.3 are a special case. See also section 9 of the survey \cite{AA}.

The paper is organized as follows. In section \ref{sect 2} we recall results from earlier work and give a snake graph interpretation of the convergents of the continued fraction as well as for the Euclidean division algorithm. We study palindromic snake graphs in section \ref{sect 3} and Markov numbers in section \ref{sect 4}.

{\emph{Acknowledgements}:} We would like to thank Keith Conrad for helpful comments.

\section{Continued fractions in terms of snake graphs}\label{sect 2}

A \emph{continued fraction} is an expression of the form
\[[a_1,a_2,\ldots,a_n]= a_1+\cfrac{1}{a_2+\cfrac{1}{\ddots +\cfrac{1}{a_n}}}\]
where the $a_i$ are  integers (unless stated otherwise) and  $a_n\ne 0$. A continued fraction is called \emph{positive} if each $a_i$ is a positive integer, and it is called \emph{even} if each $a_i$ is a nonzero even (possibly negative) integer. 
 A continued fraction $[a_1,a_2,\ldots,a_n]$ is called \emph{simple} if for each $i\ge 1$ we have $a_i\ge 1$ and $a_1$ is an arbitrary integer.

In this paper,   continued fractions are positive unless stated otherwise. Even continued fractions and their snake graphs have been studied in \cite{LS6}, and we recall some of their results below.

\subsection{The snake graph of a continued fraction}
Following \cite{CS4}, for every continued fraction $\cfa$, we construct  a snake graph $\cfg$ in such a way that the number of perfect matchings of the snake graph is equal to the numerator of the continued fraction. A \emph{perfect matching} of a graph is a subset $P$ of the set of edges such that every vertex of the graph is incident to exactly one edge in $P$.

Recall that a {\em snake graph} $\calg$ is a connected planar graph consisting of a finite sequence of tiles $G_1,G_2,\ldots, G_d$ with $d \geq 1,$ such that
$G_i$ and $G_{i+1}$ share exactly one edge $e_i$ and this edge is either the north edge of $G_i$ and the south edge of $G_{i+1}$ or the east edge of $G_i$ and the west edge of $G_{i+1}$,  for each $i=1,\dots,d-1$. See Figure~\ref{fig FourTiles} for a complete list of snake graphs with 4 tiles.
 We denote by  $\calgSW$ the 2 element set containing the south and the west edge of the first tile of $\calg$ and by $\calgNE$ the 2 element set containing the north and the east edge of the last tile of $\calg$. 
A snake graph $\calg$ is called {\em straight} if all its tiles lie in one column or one row, and a snake graph is called {\em zigzag} if no three consecutive tiles are straight.
 We say that two snake graphs are \emph{isomorphic} if they are isomorphic as graphs.

\begin{figure}
\begin{center}
\scriptsize\scalebox{0.7}{ 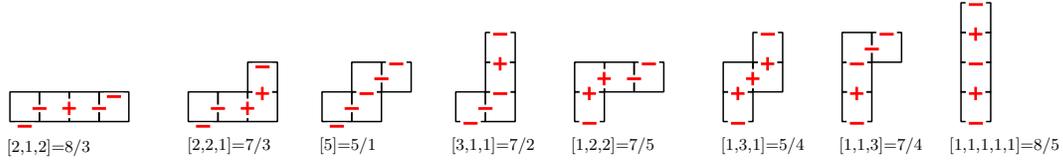}
 \caption{The snake graphs with 4 tiles together with their sign sequences and continued fractions.} \label{fig FourTiles}
\end{center}
\end{figure}

A {\em sign function} $f$ on a snake graph $\calg$ is a map $f$ from the set of edges of $\calg$ to the set $\{ +,- \}$ such that on every tile in $\calg$ the north and the west edge have the same sign, the south and the east edge have the same sign and the sign on the north edge is opposite to the sign on the south edge. 
The 
snake graph $\calg$ is determined by a sequence of tiles $G_1,\ldots,G_d$ and a sign function $f$ on the interior edges $e_1,\ldots,e_{d-1}$ of $\calg$.   Denote by $e_0\in\calgSW$ the south edge of the first tile and choose an edge $e_d\in\calgNE$. Then we obtain a sign sequence 
\begin{equation}
 \label{seq}
 (f(e_0) , f(e_1) ,\ldots ,f(e_{d-1}), f(e_d)).
\end{equation}
This sequence uniquely determines the snake graph together with a choice of a northeast edge $e_d\in \calgNE$. 

\medskip

Now let $\cfq$ be a continued fraction with all $a_i\ge 1$, and let $d= a_1+a_2+\cdots +a_n -1$.
Consider the following sign sequence
\begin{equation}
 \label{eqsign} 
\begin{array}{cccccccc}
  ( \underbrace{ -\ze,\ldots,-\ze},&  \underbrace{ \ze,\ldots,\ze},&  \underbrace{ -\ze,\ldots,-\ze},& \ldots,&  \underbrace{\pm\ze,\ldots,\pm\ze}) ,  \\
 a_1 & a_2 & a_3&\ldots&a_n
\end{array} 
\end{equation}
where  $\ze \in \{ -,+ \}$, $-\ze=\left\{\begin{array}{ll} + &\textup{if $\ze=-$;}\\ - &\textup{if $\ze=+$.}\end{array}\right. $  We define  $\sgn(a_i) = \left\{\begin{array}{ll} -\ze &\textup{if $i$ is odd;}\\ \ze &\textup{if $i$ is even.}\end{array}\right. $
 Thus each integer $a_i$ corresponds to a maximal subsequence of constant sign $\sgn(a_i)$ in the sequence (\ref{eqsign}). 
We let $\ell_i$ denote  the position of the last term in the $i$-th subsequence, thus  
$\ell_i = \sum_{j=1}^{i} a_{j}.$

 The snake graph
 $\calg[a_1,a_2,\ldots,a_n] $ of the continued fraction $[a_1,a_2,\ldots,a_n]$ is the snake graph with $d$ tiles determined by the sign sequence (\ref{eqsign}). 
 Examples are given in Figures \ref{fig FourTiles} and \ref{fig 8437}.
 Note that the two choices of the edge $e_d$ in $\calgNE$ will produce the two continued fractions $\cfa$ and $[a_1,a_2,\ldots,a_n-1,1]$; however these two continued fractions correspond to the same rational number, since $\frac{1}{a_n-1+\frac{1}{1}}=\frac{1}{a_n}$.

\begin{figure}
\begin{center}
 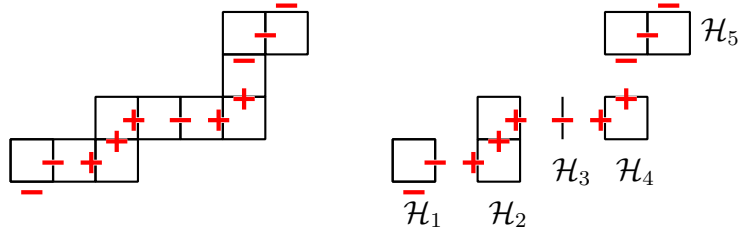
 \caption{The snake graph of the continued fraction $[2,3,1,2,3]$. The sign sequence $(-,-,+,+,+,-,+,+,-,-,-)=(f(e_0),f(e_1),\ldots, f(e_{10}))$ is given in red.  The sign changes occur in the tiles $G_{\ell_i}$, with $\ell_i=2,5,6,8$. The subgraphs $\calh_1,\ldots,\calh_5$, on the right are obtained by  removing the boundary edges of the tiles $G_{\ell_i},$  for $i=1,\ldots, n-1.$} \label{fig 8437}
\end{center}
\end{figure}

\smallskip

The following theorem is the key result of \cite{CS4}. It gives a combinatorial realization of continued fractions as quotients of cardinalities of sets.
\begin{thm}\cite[Theorem 3.4]{CS4}\label{thm1} 
 If  $m(\calg)$ denotes the number of perfect matchings of $\calg$ then
 \[ \cfa =\frac{m(\calg\cfa)}{m(\calg[a_2,\ldots,a_n])},\]
 and the right hand side is a reduced fraction. 
\end{thm}
\begin{example} In Figure \ref{expm}, we show the set of all perfect matchings of several snake graphs. 
\begin{figure}
\begin{center}
\scalebox{0.8}{ 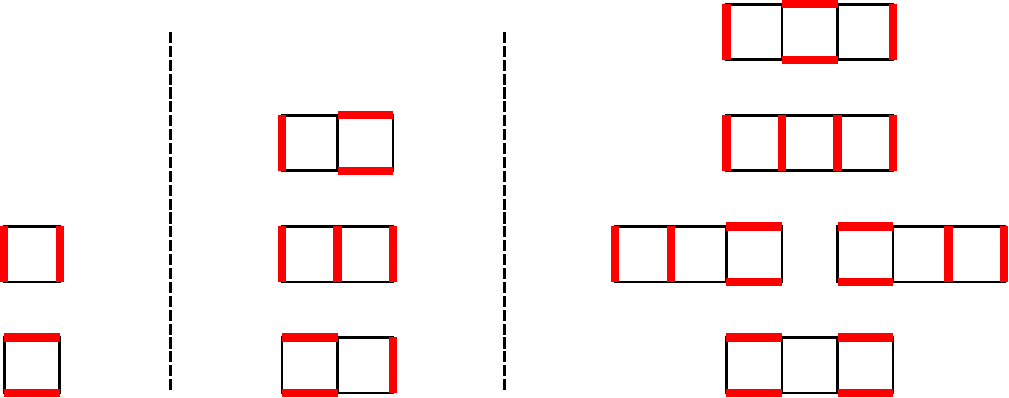}
 \caption{Small snake graphs and their perfect matchings, $\calg[2]$ (left), $\calg[3]$ (center), $\calg[2,2]$ (right)}\label{expm}
\end{center}
\end{figure}
\end{example}

\subsection{Convergents}
 Then \emph{$n$-th convergent} of the continued fraction
 $[a_1,a_2,\ldots, a_s]$ is the continued fraction 
$[a_1,a_2,\ldots, a_n]$, for $1\le n\le s$. 
By Theorem \ref{thm1}, we have that the numerator of $n$-th convergent is the number of perfect matchings of the initial segment $\calg[a_1,a_2,\ldots,a_n]$ of $\calg [a_1,a_2,\ldots,a_s]$ and the denominator of the $n$-th convergent is the number of perfect matchings of $\calg[a_2,\ldots,a_n]$.
Define \[p_1=a_1\quad,\quad p_2=a_2p_1+1\quad,\quad p_n=a_np_{n-1}+p_{n-2}, \textup{ for $n\ge 3$,} \] and 
\[ q_1=1 \quad,\quad q_2=a_2\quad,\quad q_n=a_nq_{n-1}+q_{n-2}, \textup{ for $n\ge 3$}.\]

It is well known that the $n$-th convergent of the continued fraction $[a_1,a_2,\ldots,a_s]$ is equal to $p_n/q_n$, see for example \cite[Theorem 149]{HW}.

\begin{example} \label{ex8437} The continued fraction
 $[2,3,1,2,3]=\frac{84}{37} $ has convergents 
 \[ [2,3,1,2,3]=\frac{84}{37}  \qquad [2,3,1,2]=\frac{25}{11} \qquad [2,3,1]=\frac{9}{4} \qquad [2,3]=\frac{7}{3}\qquad [2]=\frac{2}{1}\]
  
   The corresponding snake graphs are the following.\\ 
\begin{center} \scriptsize  \scalebox{0.7}{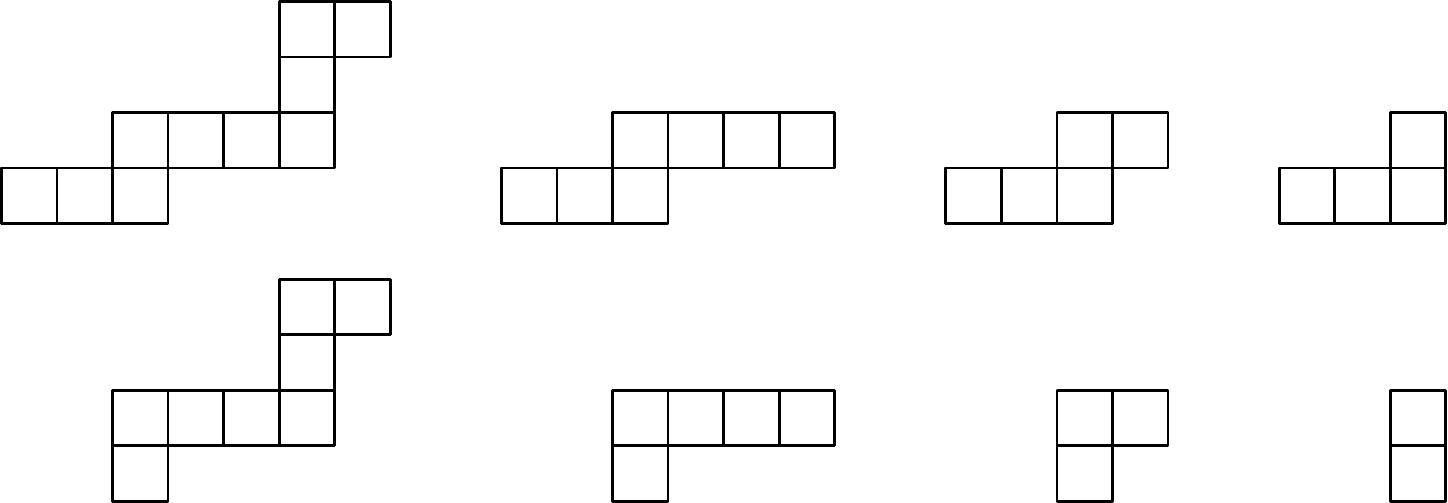}\end{center}
  where the number in the tile $G_i$ indicates the number of perfect matchings of the  subsnake graph given by the first $i$ tiles. 
  The top row shows the snake graphs corresponding to the numerators of the convergents, and the bottom row those corresponding to the denominators. Note that all the snake graphs in the top row are initial segments of the first snake graph in that row. The snake graphs in the second row are obtained from those in the first row by removing the 4 vertices of  the initial tile and all incident edges.
\end{example}

\subsection{Combinatorial realization of division algorithm}

In this section we illustrate division algorithm in terms of snake graphs for positive continued fractions and also for even continued fractions. 

In  \cite{LS6}  it is shown that if the same rational number is represented as a positive continued fraction $\cfa$ and as an even continued fraction $\cfb$ then the snake graphs associated to $\cfa $ and $\cfb$ are isomorphic.

\begin{remark}[\cite{LS6}]
 \label{rem:: [LS]}
 Let $p,q$ be relatively prime integers with $p>q>0$. Then
\begin{itemize}
\item [\rm{(a)}] if $p $ or $q$ is even, then $p/q$ has a unique even continued fraction expansion.
\item [\rm{(b)}]
 if $p$ and $q$ are both odd then $p/q$ does not have an even continued fraction expansion.
\end{itemize}
\end{remark}

\begin{example}
Let us compute the continued fraction of Example \ref{ex8437}.  The Euclidean algorithm on the left gives the continued fraction $[2,3,1,2,3]=84/37$. The algorithm on the right gives the even continued fraction $[2,4,-4,2,-2]=84/37$.
 \[ 
\begin{array}
 {rclcrcl} 
 84 &=& 2\cdot 37 +10 &\quad\quad & 84 &=& 2\cdot 37 +10 \\
 37 &=& 3\cdot 10 +7 &\quad & 37 &=& 4\cdot 10 +(-3) \\
 10 &=& 1\cdot 7 +3&\quad & 10 &=& (-4)(-3)+(-2)  \\
 7 &=& 2\cdot 3 +1 &\quad & -3 &=& 2(-2)+1\\
 3 &=& 3\cdot 1 &\quad & -2 &=&(-2)1 \\
\end{array}
 \]

 We point out that the remainders can also be realized as numbers of perfect matchings of subgraphs of the snake graph if one starts counting at the north east end of the snake graph and the division algorithm can be seen as a sequence of identities of snake graphs as follows.
\begin{center}  \scriptsize\scalebox{0.7}{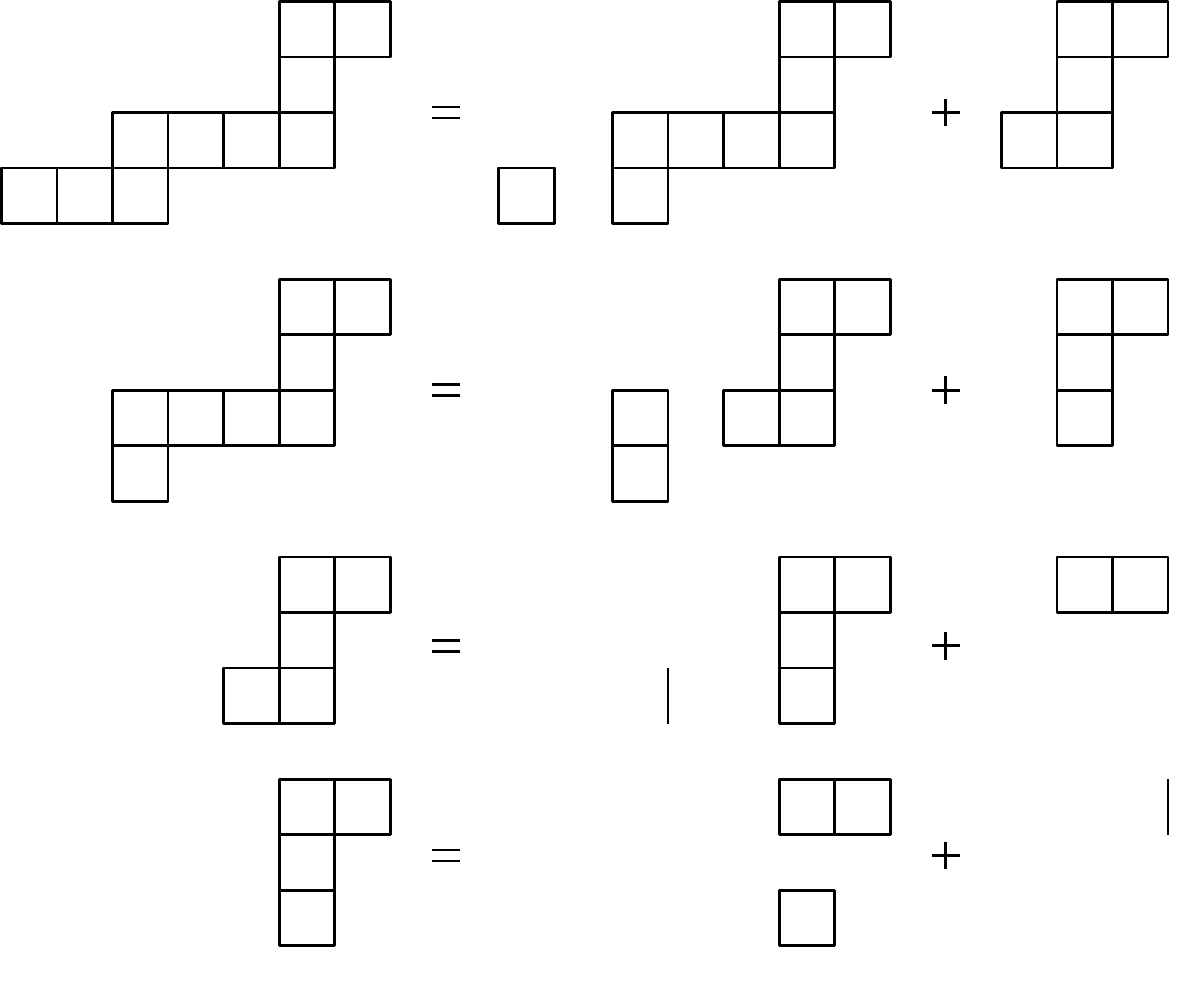}\end{center}

These identities mean that there is a bijection between the sets of perfect matchings of the snake graphs on either side. For the even continued fraction, the computation has the following realization in terms of snake graphs. \\
\begin{center} \tiny \scalebox{0.7}{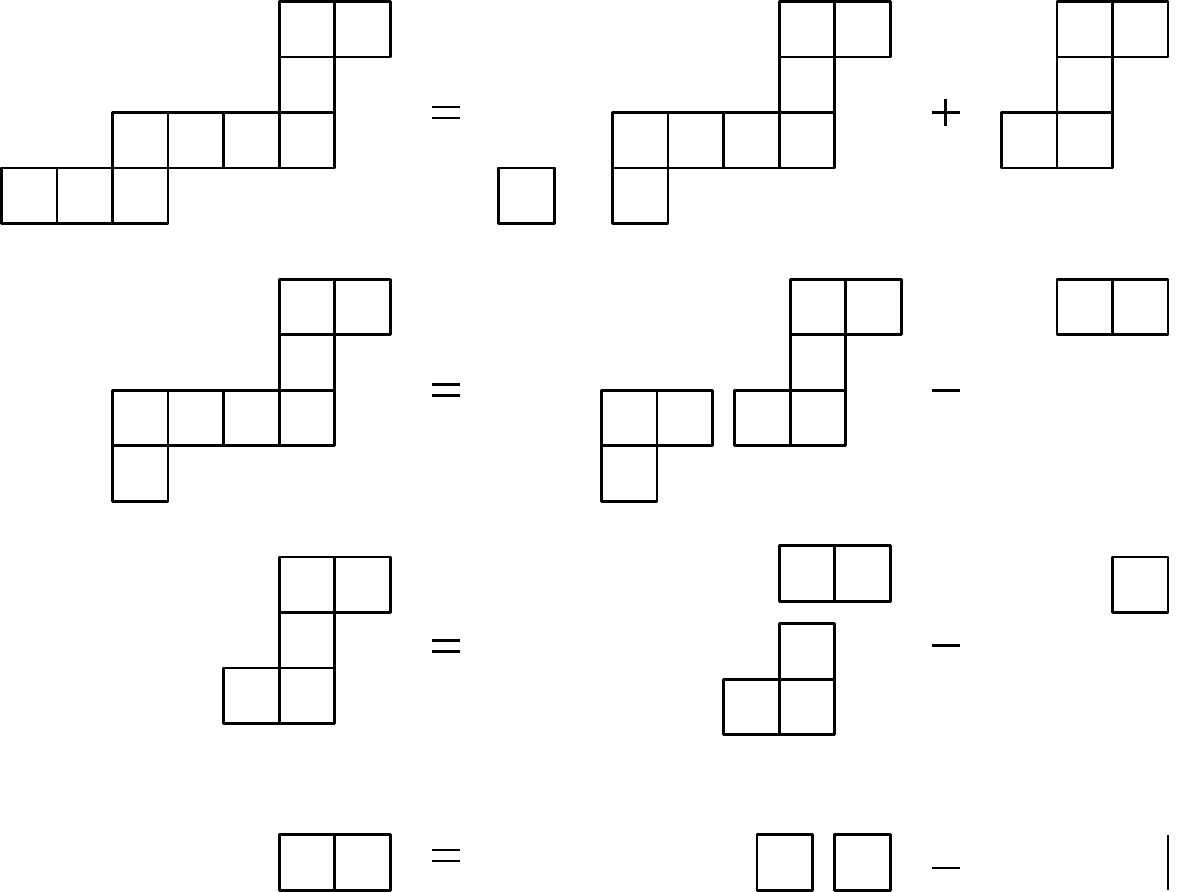}\end{center}
\end{example}
\smallskip

\section{Palindromification}\label{sect 3} 
The rotation of a snake graph by 180  degrees as well as the flips at the lines $y=x$ and $y=-x$ produce isomorphic snake graphs. We start by describing the action of these isomorphisms on the continued fractions. First note that we have an equality 
\[\calg\cfa = \calg[a_1,a_2,\ldots,a_{n}-1,1],\] 
because the change in the continued fraction corresponds simply to changing the choice of the edge $e_d\in\calgNE.$ 

Note that if $a_n=1$, then in the above equation, the coefficient $a_n-1$ is zero.  However, it is easy to see that $[a_1,\ldots, a,0,b,\ldots,a_n]=[a_1,\ldots, a+b,\ldots,a_n]$.\footnote{This follows from $a+\frac{1}{0+\frac{1}{b}}=a+b$.} Thus if $a_n=1$ then $\calg[a_1,a_2,\ldots,a_{n}-1,1]=\calg[a_1,a_2,\ldots,a_{n-1}+1].$

\begin{prop}\label{propiso}
We have the following isomorphisms.  
\begin{itemize}
\item [(a)] The flip at $y=x$.  
\[\calg\cfa\cong\calg[1,a_1-1,a_2,\ldots,a_n].\]
\item [(b)] The flip at $y=-x$.
\[\calg\cfa\cong\left\{\begin{array}{ll}
\calg[{1,a_{n}-1},\ldots,a_2,a_1] &\textup{if $e_d$ is north;}\\
\calg\caf&\textup{if $e_d$ is east.}\end{array}\right.\]
\item [(c)] The rotation by 180 degrees.
\[\calg\cfa\cong\left\{\begin{array}{ll}\calg\caf &\textup{if $e_d$ is north;}\\
\calg[1,a_n-1,\ldots,a_2,a_1]&\textup{if $e_d$ is east.}
\end{array}\right.\]
\end{itemize}
\end{prop}

\begin{proof} 
\begin{itemize}
\item [(a)]  Rotating the snake graph $\calg\cfa$ associated to a continued fraction $\cfa$ along the line $y=x$ will change its sign function 
\begin{equation*}
\begin{array}{cccccccc}
  ( \underbrace{ -\ze,\ldots,-\ze},&  \underbrace{ \ze,\ldots,\ze},&  \underbrace{ -\ze,\ldots,-\ze},& \ldots,&  \underbrace{\pm\ze,\ldots,\pm\ze})\\
 a_1 & a_2 & a_3&\ldots&a_n
\end{array} 
\end{equation*}
to the sign function
\begin{equation*} 
\begin{array}{ccccccccc}
  ( \underbrace{-\ze}, & \underbrace{\ze,\ldots,\ze},&  \underbrace{ -\ze,\ldots,-\ze},&  \underbrace{ \ze,\ldots,\ze},& \ldots,&  \underbrace{\mp\ze,\ldots,\mp\ze})\\
1& a_1-1 & a_2 & a_3&\ldots&a_n
\end{array} 
\end{equation*}
if $a_1\neq 1$ and vice versa if $a_1=1$ so the isomorphism of snake graphs follows.

\item [(b)]
 Let $\calg'$ be the image of $\calg$ under the flip. Denote by $e_1',\ldots,e_{d-1}'$ the interior edges of $\calg'$ and by $e_0'$ the south edge of the first tile. The flip is an isomorphism of graphs $\calg\to\calg'$ that maps the last tile of $\calg$ to the first tile of $\calg'$. Moreover, if $e_d$ is the east edge of the last tile in $\calg$ then it is mapped to the south edge $e_0'$ of the first tile in $\calg'$, and we have $\calg'=\calg[a_n,\ldots,a_2,a_1]$. On the other hand, if $e_d$ is the north edge of the last tile in $\calg$ then it is mapped to the west edge of the first tile in $\calg'$ and we have $\calg'=\calg[a_n,\ldots,a_2,a_1]$.

\item [(c)] Similar to part (b).\qedhere
\end{itemize}
\end{proof}

In particular, we obtain a new proof of the following classical result.
\begin{cor}
 The continued fractions $\cfa$ and $\caf$ have the same numerator.
\end{cor}
\begin{example}
We illustrate the result on the snake graph $\calg=\calg[2,2,1]=\calg[2,3]$.
We see from Figure \ref {fig FourTiles} that the flip at $y=x$ produces  $\calg[1,1,3]$, the
 flip at $y=-x$ produces  $\calg[3,1,1]$ and the rotation produces $\calg[1,2,2]$.
 
Note that $e_d$ is north in $\calg[2,2,1]$ and east in $\calg[2,3]$. 
Therefore from the formulas in the proposition, we have the desired result.
\end{example}
 
\begin{definition}
 A
 continued fraction $[a_1,a_2,\ldots,a_n]$  is said to be {\em of even length} if $n$ is even. It is called {\em palindromic} if the sequences $(a_1,a_2,\ldots,a_n)$ and $(a_n,\ldots,a_2,a_1)$ are equal.

A snake graph $\calg$ is called {\em palindromic} if it is the snake graph of a palindromic continued fraction. Moreover $\calg$ is called {\em palindromic of even length}  if it is the snake graph of a palindromic continued fraction of even length.

A snake graph $\calg$ has a {\em rotational symmetry at its center tile} if $\calg$ has a tile $G_i$, such that the rotation about 180\/$^\circ$ at the center of this tile is an automorphism of $\calg$. \end{definition}

See Figure \ref{fig rotation} for examples.

\begin{remark} 
 To avoid confusion, we emphasize that if a snake graph   $\calg\cfa$   is palindromic of even length, the word ``even'' refers to the number $n$ of entries in the continued fraction. Thus ``even length'' means that the continued fraction is of even length and \emph{not} that the snake graph has an even number of tiles. Actually we shall show below that palindromic snake graphs of even length always have an odd number of tiles.
\end{remark}
\begin{remark} If a snake graph $\calg$ has a rotational symmetry at its center tile then
  the number of tiles of $\calg$  must be odd and the center tile is the $i$-th tile of $\calg$, where $i=(d+1)/2$ and $d$ is the total number of tiles. 
\end{remark}

\begin{thm}
 \label{thm pal1}
 A snake graph is palindromic of even length if and only if it has a rotational symmetry at its center tile.
\end{thm}

\begin{proof}
 Suppose first that $\calg$ is a palindromic snake graph  of even length with corresponding  palindromic continued fraction
$[a_1,\ldots,a_n,a_n,\ldots,a_1]$,  and let $d$ be the number of tiles of $\calg$. 
From the construction in section \ref{sect 2}, we know that the sign sequence determined by the continued fraction has length $d+1=2(a_1+a_2+\cdots+a_n)$. This shows that $d$ is odd. Let $i=(d+1)/2$, and let $G_i$ be the $i$-th tile of $\calg$. Thus $G_i$ is the center tile of $\calg$. Therefore, the initial subgraph consisting of the first $i-1$ tiles of $\calg$ is isomorphic to $\calg [a_1,a_2,\ldots,a_n]$,  the terminal subgraph consisting of the last $i-1$ tiles of $\calg$ is isomorphic to  $\calg[a_n,\ldots,a_2,a_1] $
and  the 3 consecutive tiles $G_{i-1},G_i,G_{i+1}$ form a straight subgraph. Therefore the two interior edges $e_{i-1}$ and $e_{i}$ are parallel. Note that  the edge $e_{i-1}$ is the last interior edge of $\calg$ that belongs to the initial subgraph  $\calg[a_1,a_2,\ldots,a_n]$ and $e_{i}$ is the first interior edge of $\calg$ that belongs to the terminal subgraph  $\calg[a_n,\ldots,a_2,a_1]$. Thus $e_{i-1}$ and $e_{i}$ being parallel implies that $e_0$ and $e_d$ are parallel as well. Since $e_0$ is the south edge of $\calgSW$ it follows that   $e_d$ is the north edge in $\calgNE$. Now Proposition \ref{propiso} implies that  the rotation by 180$^\circ$ at the center of the tile $G_i$ is an automorphism of $\calg$.

Conversely, if $\calg$ has a rotational symmetry at its center tile $G_i$, then  the 3 consecutive tiles $G_{i-1},G_i,G_{i+1}$ form a straight subgraph, and thus the sign changes from the interior edge $e_{i-1}$ to the interior edge $e_i$. Consequently, if the subgraph given by the first $i-1$ tiles of $\calg$ is of the form $\calg[a_1,a_2,\ldots, a_s]$ and the subgraph given by the last $i-1$ tiles of $\calg$ is of the form $\calg[a_{s+1},\ldots, a_n]$, then because of the sign change at the tile $G_i$, we conclude that $\calg=\calg[a_1,\ldots,a_s,a_{s+1},\ldots,a_n]$.
The rotational symmetry now implies that  the sequences $(a_1,a_2,\ldots,a_s)$ and $(a_n, a_{n-1},\ldots, a_{s+1})$ are equal, as long as we choose the edge $e_d\in\calgNE$ to be the north edge such that $e_d$ is the image of the edge $e_0$ under this rotation. 
This shows that $\calg$ is palindromic of even length.
 \end{proof}

\begin{figure}
\begin{center}
\scalebox{0.7}{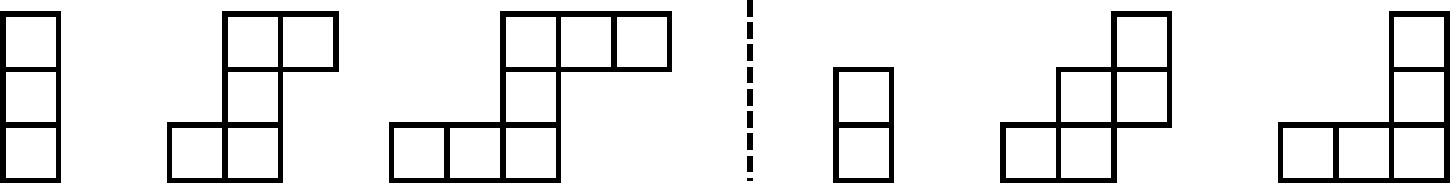}
\caption{The three snake graphs on the left are rotationally symmetric at their center tile. The three snake graphs on the right are not. }
\label{fig rotation}
\end{center}
\end{figure}
\begin{example}
 The continued fractions of the 3  snake graphs on the left of Figure~\ref{fig rotation} are $[1,1,1,1], [3,3]  $ and $[2,2,2,2]$, respectively. The snake graphs are rotationally symmetric at their center tile and the continued fractions are palindromic and of even length. On the other hand, the continued fractions of the  3  snake graphs on the right of Figure~\ref{fig rotation} are $[1,1,1], [6] $ and $[2,2,2]$, respectively. Each of these is palindromic but none is of even length. Accordingly, the snake graphs have no center tile rotational symmetry.
\end{example}
 
 Given a snake graph $\calg$, we can construct a palindromic snake graph of even length $\calg_{\leftrightarrow}$ by glueing two copies of $\calg$ to  a new center tile. This graph is called the {\em palindromification} of $\calg$.
More precisely, if $\calg=\calg[a_1,a_2,\ldots,a_n]$, then its palindromification is the snake graph $ \calg_{\leftrightarrow} =\calg[a_n,\ldots, a_2,a_1,a_1,a_2,\ldots,a_n]$.

\begin{thm} \label{thm palsnake}
 Let $\calg=\calg[a_1,a_2,\ldots,a_n]$ be a snake graph  and $\calg_\leftrightarrow$ its palindromification. Let $\calg'=\calg[a_2,\ldots,a_n]$.
 Then 
\[m(\calg_{\leftrightarrow}) = m(\calg)^2 + m(\calg')^2.\]
\end{thm}

\begin{proof}
By definition, we have 
 $ \calg_{\leftrightarrow} =\calg[a_n,\ldots, a_2,a_1,a_1,a_2,\ldots,a_n]$. 
  Using
 \cite[Theorem 5.1]{CS4}, we see that multiplying with the single edge $b_n$ gives 
  \[b_n\calg_{\leftrightarrow}=\calg[a_n,\ldots,a_1]\,\calg[a_1,\ldots,a_n] + \calg[a_n,\ldots,a_{2}]\,\calg[a_{2},\ldots,a_n]\]
   which, by symmetry, is equal to
 \[\calg[a_1,\ldots,a_n]^2 + \calg[a_2,\ldots,a_{n}]^2 = \calg\calg + \calg'\calg'.
 \]
The result now follows by counting perfect matchings on both sides.
\end{proof}
 
\begin{remark}\label{rem 310}
The following identity can be proved in almost the same way.
 \begin{eqnarray*}\lefteqn{b_n\ \calg[a_{n-1},\ldots, a_2,a_1,a_1,a_2,\ldots,a_n]} \\
&=& \calg[a_{n-1},\ldots, a_2,a_1]\calg[a_1,a_2,\ldots,a_n]+ 
 \calg[a_{n-1},\ldots, a_2]\calg[a_2,\ldots,a_n]
 .\end{eqnarray*}
\end{remark}
 
 
As a  consequence, we obtain a new proof of the following known result by specializing the weights of the snake graphs to the integer 1.
\begin{cor}
 \label{thm pal2}
 Let $[a_1,a_2,\ldots,a_n] = \cfrac{p_n}{q_n} $. Then 
\[ [a_n,\ldots, a_2,a_1,a_1,a_2,\ldots,a_n] =\frac{ p_n^2 +q_n^2}{p_{n-1}p_n+q_{n-1}q_n}.\]
Moreover, the expression on the right hand side is a reduced fraction.
\end{cor}

\begin{proof}
 The left hand side is equal to $\cfrac{m(\calg_{\leftrightarrow})}{m(\calg[a_{n-1},\ldots, a_2,a_1,a_1,a_2,\ldots,a_n])} $ and the right hand side is equal to 
 \[\cfrac{m(\calg)m(\calg)+m(\calg')m(\calg')}{  m( \calg[a_1,a_2,\ldots,a_{n-1}]) M( \calg)+m(\calg[a_2,\ldots,a_{n-1}])m( \calg')} \] 
 The numerators of these two expressions are equal by Theorem \ref{thm palsnake}, and the denominators are equal by Remark \ref{rem 310}.
\end{proof}
\begin{example}
%
$[2,1,3]=\cfrac{11}{4}$, $[2,1]= \cfrac{3}{1} $ and $[3,1,2,2,1,3]=\cfrac{137}{37} =\cfrac{11^2+4^2}{3\cdot 11+1\cdot 4}$.
\end{example}

\subsection{Sums of two relatively prime squares}
An integer $N$ is called a {\em sum of two relatively prime squares} if there exist integers $p>q\ge 1$ with $\gcd(p,q)=1$ such that $N=p^2+q^2$. 
We have   the following corollary.
\begin{cor}
 \label{cor squares}
 (a) If $N$ is a sum of two relatively prime squares then there exists a palindromic snake graph of even length  $\calg$ such that $m(\calg)=N$. 
\\
\indent (b) For each positive integer $N$, the number of ways one can write $N$ as a sum of two relatively prime squares is equal to one half of the number of palindromic snake graphs of even length with $N$ perfect matchings.\\
\indent (c)  For each positive integer $N$, the number of ways one can write $N$ as a sum of two relatively prime squares is equal to one half of the number of  palindromic continued fractions of even length with numerator $N$.
\end{cor}
\begin{proof}
 Let  $p>q\ge 1$ be such that $N=p^2+q^2$ and $\gcd(p,q)=1$, and let  $[a_1,a_2,\ldots,a_n] = \frac{p}{q} $ be the continued fraction of the quotient.
 Combining Theorem~\ref{thm pal2} and Theorem \ref{thm1}, we see that $\calg[a_n,\ldots, a_2,a_1,a_1,a_2,\ldots,a_n]$ has $N$ perfect matchings. This shows part (a). Parts (b) and (c) follow from  the bijections of \cite[Theorem 4.1]{CS4}.
 \end{proof}

\begin{example} The integer $5$ can be written uniquely as sum of two relatively prime squares as $5=2^2+1^2$. 
 The even length palindromic continued fractions  with numerator $5$ are $[2,2]$ and $[1,1,1,1]$, corresponding to the snake graphs \scalebox{0.4}{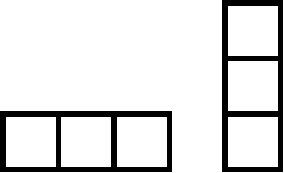}\,, respectively.
\end{example}

 \subsection{Odd palindromes}
 For odd length palindromic continued  fractions we have a similar result describing the numerator as a difference of two squares.
 
\begin{thm}
 \label{thm oddpal}
 Let $[a_1,a_2,\ldots, a_n]=\cfrac{p_n}{q_n}$ and $[a_2,a_3,\ldots, a_n]=\cfrac{q_n}{r_n}$. Then
 \[ [a_n,\ldots, a_2,a_1,a_2,\ldots,a_n] =\frac{ p_n^2 -r_{n}^2}{p_{n-1}p_n-r_{n-1}r_n}.\]

\end{thm}
\begin{remark}
 The fraction on the right hand side of the equation is not reduced. The greatest common divisor of the numerator and the denominator is $a_1$. 
\end{remark}
 
\begin{proof} Because of Theorem~\ref{thm1},
we have \[[a_n,\ldots,a_2,a_1,a_2,\ldots,a_n]=\cfrac{m(\calg[a_n,\ldots,a_2,a_1,a_2,\ldots,a_n])}{m(\calg[a_{n-1},\ldots,a_2,a_1,a_2,\ldots,a_n])}.\]
 Using part(b) of \cite[Theorem 5.1]{CS4}
  with $i=n$ and $j=0$, we see that the numerator is equal to 
 \[\cfrac{m(\calg[a_n,\ldots,a_1])\,m(\calg[a_1,\ldots,a_n]) - m(\calg[a_n,\ldots,a_{3}])\,m(\calg[a_{3},\ldots,a_n]) }
{m(\calg{[a_1]}) }
 \]
 which, by symmetry, is equal to
 \[\cfrac{p_n^2 - r_n^2 }
{a_1}.
 \]
 Moreover, this is an integer, since it is the number of perfect matchings of a snake graph.
 On the other hand, again using \cite[Theorem 5.1]{CS4}, 
but now with $i=n-1$ and $j=0$, the denominator is equal to 
  \[\cfrac{m(\calg[a_{n-1},\ldots,a_1])\,m(\calg[a_1,\ldots,a_n]) - m(\calg[a_{n-1},\ldots,a_{3}])\,m(\calg[a_{3},\ldots,a_n]) }
{m(\calg{[a_1]}) }
 \]
 which is equal to
 \[\cfrac{p_{n-1}p_n- r_{n-1}r_n}
{a_1}.
 \]
This proves the theorem.
\end{proof}
\begin{example}
 Let $[a_1,a_2,a_3]=[2,1,3]=11/4=p_n/q_n$, and thus $q_n/r_n=[1,3]=4/3$, $p_{n-1}=3$ and $r_{n-1}=1$. Then  $[3,1,2,1,3] = \frac{56}{15}$  and on the other hand, $\frac{11^2-3^2}{ 3\times11-1\times3} = \frac{112}{30}=\frac{56}{15}$.
\end{example}

\subsection{Squares and almost palindromes} One can also realize the square of the numerator of a continued fraction $[a_1,a_2,\ldots,a_n]$  as the numerator of a continued fraction that is almost palindromic.  
\begin{thm}
\label{thm squares} Let $\cfa=\frac{p}{q}.$    Then \[ [a_1,\ldots,a_{n-1},a_n+1,a_n-1,a_{n-1},\ldots, a_1]=\frac{p^2}{pq+(-1)^n}.\]
\end{thm}
\begin{proof} It follows from  the  formula for grafting with a single edge of \cite[section 3.3 case 3]{CS2}  that the numerator of the left hand side is equal to 
 \[ \caln[a_1,\ldots,a_{n-1},a_n]\,\caln[a_n-1,a_{n-1},\ldots, a_1]+\caln[a_1,\ldots, a_{n-1}]\, \caln[a_n,a_{n-1},\ldots,a_1],\]
 which can be written as
 \[p\,( \caln[a_1,\ldots, a_{n-1},a_{n}-1]+\caln[a_1,\ldots, a_{n-1}]),\]
 and from the recursive definition of convergents it follows that the term in parentheses is also equal to $p$. 
 This shows that the numerators on both sides agree.
 
 The denominator on the left hand side is equal to $\caln[a_2,\ldots,a_{n-1},a_n+1,a_n-1,a_{n-1},\ldots, a_1]$ and using  the grafting with a single edge formula again, this is equal to 
 \[\caln [a_2,\ldots,a_{n-1},a_n]\,\caln[a_n-1,a_{n-1},\ldots, a_1]+\caln[a_2,\ldots, a_{n-1}]\,\caln[a_n,a_{n-1},\ldots,a_1].\] 
 Now using part (b) of \cite[Theorem 5.2]{CS4} 
 with $i=2$ and $i+j=n-1$ on the second summand, we see that the above expression is equal to 
 \[\caln [a_2,\ldots,a_{n-1},a_n]\,\caln[a_n-1,a_{n-1},\ldots, a_1]+\caln[a_1,\ldots, a_{n-1}]\,\caln[a_2,\ldots,a_n] +(-1)^n\]
 \[=\caln [a_2,\ldots,a_{n-1},a_n]\left( \caln[a_1,\ldots,a_{n-1},a_n-1]+\caln[a_1,\ldots, a_{n-1}]\right) +(-1)^n \]
  \[=qp  +(-1)^n .\qedhere\]
\end{proof}

\begin{example}
$\frac{11}{4}=[2,1,3]$
and 
$[2,1,4,2,1,2]=\frac{121}{43}=\frac{11^2}{4\cdot11-1}$.\end{example}

\begin{remark} Removing the first and last coefficient of a palindromic continued fraction gives a new palindromic continued fraction. The relation between the two is well-known, see \cite[Section 9]{V}. More precisely, let $\frac{p}{q}=\cfa$ then $\cfa$ is palindromic if and only if $p$ divides $q^2+(-1)^n$. Moreover, in this case, the quotient $\frac{q^2+(-1)^n}{p}$ is the numerator of the palindromic continued fraction $[a_2,a_3,\ldots,a_{n-1}]$.
\end{remark}

\begin{example}
\[\begin{array}{ccccccccc}
 [2,1,1,2,2,1,1,2]=194/75
&\qquad  & \frac{75^2+1}{194} = 29 & \\
 \\
 
 [1,1,2,2,1,1]=29/17  && \frac{17^2+1}{29}=10\\
 \\
 
 [1,2,2,1]=10/7  & &\frac{7^2+1}{10}=5\\
\\

 [2,2]=5/2  \\
\end{array}\]
\end{example}

 \section{Markov numbers}\label{sect 4}
 A triple of positive integers $(m_1,m_2,m_3)$
is called a {\em Markov triple} if it is a solution to the Markov equation 
\[x^2+y^2+z^2=3xyz.\]
 An integer is called  a {\em Markov number} if it is a member of a Markov triple. Frobenius conjectured in 1913 that the largest number in a Markov triple determines the other two \cite{F}. This is known as the uniqueness conjecture for Markov numbers. For an account of the history and many attempts of solving this conjecture see the monograph \cite{Aigner}. 
 
 It is known that every other Fibonacci number is a Markov number and so is every other Pell number.
 Examples of Markov triples are 
 \begin{equation}
 (1,1,1),(2,1,1),(2,5,1),(13,5,1),(13,5,194), (2897,5,194).
 \label{Markov seqn}
 \end{equation}
 
 In this sequence each triple is obtained from the previous one by the exchange relation $m_i' m_i=m_j^2+m_k^2$ or equivalently $m_i'=3m_jm_k-m_i$.

\subsection{Markov snake graphs}
\label{sec::MSG}

It has been shown in \cite{BBH,Propp} that the Markov triples  are related to the clusters of the cluster algebra associated to the torus with one puncture.  This relation is given explicitly by sending a cluster to the triple obtained by setting the three initial cluster variables equal to 1. The sequence (\ref{Markov seqn}) of Markov triples corresponds to a sequence of mutations in the cluster algebra. Since the cluster variables are computed by snake graphs, we can interpret the Markov numbers in terms of snake graphs.

Given a slope $p/q$ with $p<q$, $\gcd(p,q)=1$, there is the associated Markov number $m_{p/q}$. Take the torus with one puncture and usual covering in the plane such that the cluster variables $x_1,x_2$ correspond to the standard basis vectors $e_1,e_2$   of the plane. Let $x_3$ correspond to the line segment between the point $(1,0)$ and $(0,1)$. The line segment from $(0,0)$ to $(q,p)$ represents the cluster variable whose numerator has $m_{p/q}$ terms counting multiplicities, see Figure \ref{linepqch}.
 This line has slope $p/q$ and has a crossing pattern with the standard grid of the following form
\[\begin{array}{ccccccccc} \underbrace{x_2,\cdots,x_2,}&x_1,&  \underbrace{x_2,\cdots,x_2,}&x_1,&\cdots &x_1,& \underbrace{x_2,\cdots,x_2}\\
v_1&&v_2&&\cdots &&v_p
\end{array}\]
meaning that the line crosses $v_1$ vertical edges, then one horizontal edge, then $v_2$ vertical edges, and so on.
The $v_i$ are computed using the floor function as follows.

\[\begin{array}{rcll}
v_1&=& \displaystyle \left \lfloor \frac{p}{q} \right \rfloor \\
v_i&=& \displaystyle \left \lfloor \frac{iq}{p} \right \rfloor  -  ( v_1 + v_2 + \cdots + v_{i-1} ) \quad 
&\textup{for $i =2,\ldots, p-1$;}\\[10pt]
v_p&=& \displaystyle q -1- ( v_1 + v_2 + \cdots + v_{p-1} ) .\end{array}\]

Note that $|v_i-v_j| \le 1$, and $v_i=v_{p+1-i}$. Moreover, since the points $(0,0), (q,p)$ are the only lattice points on the line segment we have $v_1\le v_i$, for all $i$
and $v_1+\cdots+v_{1+j} \le v_{i}+\cdots v_{i+j}$ for all $i,j$.
The total number of crossings with vertical edges is $v_1+v_2+\cdots+v_p=q-1$ and the total number of crossings with horizontal edges is $p-1$. 

In the triangulation corresponding to the initial cluster $(x_1,x_2,x_3)$, the line also crosses the diagonals $x_3$. In fact, every other crossing is with $x_3$. 

Therefore, the corresponding snake graph $\calg$ has $(p-1)$ tiles for the crossings with $x_1$, $(q-1)$ tiles  for the crossings with $x_2$ and $p+q-1$ tiles  for the crossings with $x_3$; and its shape is given by the continued fraction \[ \begin{array}{ccccccccc} [2,&\underbrace{1,\cdots, 1,}&2,2,&  \underbrace{1,\cdots, 1,}&2,2, &\cdots &2,2,& \underbrace{1,\cdots,1,}&2]\\
&2(v_1-1)&&2(v_2-1)&&\cdots &&2(v_p-1)
\end{array}
\]
which has $2*(q-p-1)$ times the coefficient 1 and $2*p$ times the coefficient 2, for a total of $2q-2$ coefficients. Note that  $v_i=v_{p+1-i}$. 
We list the initial segments of continued fractions corresponding to small slope in Table \ref{table 1}.
\begin{table}\caption{Initial segments of continued fractions of Markov numbers according to their slope}
\begin{tabular} {|r|l|}
\hline
slope&continued fraction\\
\hline
$0<p/q<1/2$& [2,1,1\ldots]\\
$p/q=1/2$&[2,2]\\
$1/2<p/q<2/3$& [2,2,2,1,1\ldots]\\
$p/q=2/3$&[2,2,2,2]\\
$2/3<p/q<3/4$& [2,2,2,2,2,1,1\ldots]\\
$p/q=3/4$&[2,2,2,2,2,2]\\
$3/4<p/q<4/5$& [2,2,2,2,2,2,2,1,1\ldots]\\
[1pt]
\hline 
\end{tabular}\label{table 1}
\end{table}

The above discussion gives a new proof of the following result which is due to Frobenius.

\begin{thm}\label{thm markov}
\cite[Section 10]{F} 
 Every Markov number $m_{p/q}$ is the numerator of a palindromic continued fraction $[a_{n},\ldots, a_2,a_1,a_1,a_2,\ldots,a_n]$ of even length such that 
 \begin{enumerate}
\item $a_i\in\{1,2\}$, $a_n=2$;
\item  If $p+1=q$ then  $n=p$ and $a_i=2$ for all $i$;
\item If $p+1<q$ then $\frac{c-1}{c}<\frac{p}{q}<\frac{c}{c+1}$,  for a unique positive integer $c$ and  
\begin{enumerate}
\item 
there are   at most $p+1$  subsequences of $2$s; the first and last are  of odd  length $2c-1$ and all others  are  of even length $2c$ or $2c+2$;
\item there are  at most $p$ maximal subsequences of $1$s; each of these is of even length $2(\nu_i-1)$ and $|\nu_i-\nu_j|\le 1$ for all $i\ne j$.
\end{enumerate}
\end{enumerate}
Moreover, the resulting map \[p/q \mapsto [a_{n},\ldots, a_2,a_1,a_1,a_2,\ldots,a_n]\] from rational numbers between 0 and 1 to palindromic continued fractions is injective.
\end{thm}

The snake graphs obtained from a line on the once-punctured torus  by the above procedure are called \emph{Markov snake graphs}. These have first appeared in \cite{Propp}. In the textbook \cite{Aigner} these graphs are called domino graphs.
We can reformulate Theorem \ref{thm markov} as follows.
\begin{cor}
 Every Markov snake graph is rotationally symmetric at its center tile. Moreover
 \begin{enumerate}
\item the snake graph has exactly $p$ horizontal segments  each of which has exactly  $2(\nu_i-1)+3$ tiles, and $|\nu_i-\nu_j|\le 1$ for all $i\ne j$; 
\item the snake graph has exactly $p-1$ vertical segments each of which has exactly 3 tiles.
\end{enumerate}

\end{cor}
As a direct consequence we obtain the following.

\begin{cor}\label{cor 43}
 Every Markov number, except 1 and 2, is a sum of two relatively prime squares.
\end{cor}
\begin{proof}
 This follows directly from Theorems \ref{thm markov}  and \ref{thm pal2}.
\end{proof}

In general, the decomposition of an integer as a sum of two relatively prime squares is not unique. The smallest\footnote{Note that $50=7^2+1^2=5^2+5^2$, however the expression $5^2+5^2$ is not a sum of two relatively prime squares since $\gcd(5,5)\ne 1$.}
 example is the integer 65 which is $8^2+1^2$ and also $7^2+4^2$, and on the other hand 65 is the numerator of the continued fractions $[8,8]$ and $[3,1,1,1,1,3]$.  The smallest example among the  Markov numbers is the Fibonacci number 610, which is $23^2+9^2$ and also $21^2+13^2$. Note that $21/13=[1,1,1,1,1,2]$ and its palindromification  $[2,1,1,1,1,1,1,1,1,1,1,2]$ is a Markov snake graph (corresponding to the slope $1/7$). On the other hand, $23/9=[2,1,1,4]$ and its palindromification $[4,1,1,2,2,1,1,4]$ is not Markov.
 
We can sharpen Corollary \ref{cor 43} as follows. Whenever $b/a=\cfa$, we use the notation $\calg(b/a) $ for the snake graph $\calg\cfa$, and we call the snake graph $\calg[a_n,\ldots,a_2,a_1,a_1,a_2,\ldots,a_n]$ the palindromification of $\calg(b/a)$.

\begin{cor}\label{cor 44}
Let $m>2$ be a Markov number. Then there exist positive integers $a<b$ with $\gcd(a,b)=1$ such that $m=a^2+b^2$, $2a \le b < 3a$ and

(a) the palindromification of the snake graph $\calg(b/a)$  is a Markov snake graph;

(b) the continued fraction expansion of the quotient $b/a$ contains only $1s$ and $2s$. 

\end{cor}
\begin{proof}
 Let $p/q$ be a slope such that $m=m_{p/q}$ and let $[a_n,\ldots,a_1,a_1,\ldots,a_n]$ be the palindromic continued fraction given by Theorem \ref{thm markov}. In particular, the numerator of $[a_n,\ldots,a_1,a_1,\ldots,a_n]$ is $m$, and the snake graph $\calg[a_n,\ldots,a_1,a_1,\ldots,a_n] $ is Markov. Define $a$ and $b$ by $b/a=\cfa$ with $0<a<b$ and $\gcd(a,b)=1$.  Since $a_1=2$, we see that $2a\le b\le 3a$, and $2a=b$ if and only if $m=5$. Then Corollary \ref{cor squares} implies $m=a^2+b^2$, and this proves (a). Part (b) follows directly, since the continued fraction of any Markov snake graph contains only 1s and 2s.
\end{proof}

\begin{remark}
 Part (a) of the corollary implies part (b), because the continued fraction of every Markov snake graph contains only 1s and 2s.
\end{remark}

We conjecture that  the pair $(a,b)$ in Corollary \ref{cor 44} (a) and (b) is uniquely
determined by the Markov number.
\begin{conjecture}
\label{conj 2}
 Let $m>2$ be a Markov number. Then there exist \textbf{unique} positive integers $a<b$ with $\gcd(a,b)=1$ such that $m=a^2+b^2$, $2a \le b < 3a$ and  the palindromification of the snake graph $\calg(b/a)$  is a Markov snake graph. 
 \end{conjecture}
The following conjecture is stronger.
\begin{conjecture}
\label{conj 1}
 Let $m>2 $ be a Markov number. Then there exist {\bf unique} positive integers $a<b$ with $\gcd(a,b)=1$ such that $m=a^2+b^2$, $2a \le b < 3a$ and the continued fraction expansion of the quotient $b/a$ contains only $1s$ and $2s$. 
\end{conjecture}

\begin{remark}
 We have checked these conjectures by computer for all Markov numbers of slope $p/q$ with $p<q<70$. This is a total of 1493 Markov numbers, the largest of which is \[56790444570379838361685067712119508786523129590198509.\] This number is larger than $5.679*10^{52}$.
\end{remark}

\begin{thm}
 \label{thm conj}
(a) Conjecture \ref{conj 1} implies Conjecture \ref{conj 2}.

(b) Conjecture \ref{conj 2} is equivalent to the Uniqueness Conjecture for Markov numbers.
\end{thm}
 
\begin{proof} (a) The existence of the pair $(a,b)$ in both conjectures follows from Corollary \ref{cor 44}. 
Thus we need to show that the uniqueness in Conjecture \ref{conj 1} implies the uniqueness in Conjecture \ref{conj 2}.
Since every Markov snake graph has a continued fraction that contains only 1s and 2s, the condition in Conjecture \ref{conj 1} is weaker than the condition in Conjecture \ref{conj 2}. If the pair  $(a,b)$ is uniquely determined by the weaker condition then it is also uniquely determined by the stronger condition. This proves (a).

(b) First assume that Conjecture \ref{conj 2} holds.
 Let $m$ be a Markov number. For $m=1,2,$ the uniqueness conjecture is known, so we may assume  $m\ge 3$. We want to show that $m$ is the maximum of a unique Markov triple, or equivalently, that there is a unique slope $p/q$ with $0<p<q$ and $\gcd(p,q)=1$ such that $m=m_{p/q}$.

 For every slope $p/q$, the proof of Corollary \ref{cor 44} constructs a pair $(a,b)$  satisfying the conditions in the corollary, in particular $m_{p/q}=a^2+b^2$.

Suppose that there are two slopes $p/q$ and $p'/q'$ such that the corresponding Markov numbers $m_{p/q}$ and $m_{p'/q'}$ are equal. Denote by $[a_n,\ldots,a_2, a_1,a_1, a_2,\ldots,a_n]$ and $[a'_n,\ldots,a'_2,a'_1,a'_1,a'_2,\ldots,a'_{n'}]$ the corresponding continued fractions that are given by Theorem \ref{thm markov}. Define $a,b,a',b'$ by $b/a=\cfa$ and $b'/a'=[a_1',a_2',\ldots,a_{n'}']$. Then the two pairs $(a,b) $ and $(a',b')$ both satisfy the condition in Corollary \ref{cor 44}.
Our assumption that  Conjecture \ref{conj 1} holds implies that $(a,b)=(a',b')$. Therefore $\cfa=[a_1',a_2',\ldots,a_{n'}']$ and hence $p/q=p'/q'$, by Theorem \ref{thm markov}. Thus the uniqueness conjecture holds.

Conversely, let us now  assume that the uniqueness conjecture for Markov numbers  holds. Let $m\ge 3$ be a Markov number. We want to show that the pair $(a,b)$ of Corollary \ref{cor 44}  is unique. Assume there is another pair $(c,d)$ that satisfies the conditions of the corollary. Then $m=a^2+b^2=c^2+d^2$ and $m$ is the number of perfect matchings of the palindromifications of both $\calg(b/a)$ and $\calg(d/c)$. Since both palindromifications are Markov snake graphs, both define a slope $p/q$ and $r/s$ such that $m=m_{p/q}=m_{r/s}$. Our assumption that the uniqueness conjecture holds implies that $p/q=r/s$, and therefore the snake graphs $\calg(b/a)$ and $\calg(d/c)$ are equal. It follows that $(a,b)=(c,d)$.
\end{proof}

\subsection{Markov snake graphs in terms of Christoffel words}\label{sect: Christoffel}
We give another construction for the Markov snake graph. The line with slope $p/q$ defines a  lattice path, the {\em lower Christoffel path}, which is the lattice path from $(0,0)$ to $(q,p)$ that satisfies the conditions:
\begin{itemize}
\item[(a)]
The path lies below the line segment from $(0,0)$ to $(q,p)$.
\item[(b)] The region enclosed by the path and the line segment contains no lattice point besides those on the path.
\end{itemize}
We give an example in the upper left picture in Figure \ref{linepqch}, where the line segment is drawn in blue and  the lower Christoffel path in red.  

The Christoffel word of slope $p/q$ is obtained from the Christoffel path by writing the letter $x$ for each horizontal step, and writing the letter $y$ for each vertical step. In the example of Figure \ref{linepqch}, the Christoffel word is $xxxyxxyxxy$. We refer the reader to \cite{BLRS} for  further results on  Christoffel words.

To obtain the Markov snake  graph we use tiles of side length 1/2 and place them along the Christoffel path such that the horizontal steps of the Christoffel path become the south boundary of the snake graph and the vertical steps of the Christoffel path become the east boundary of the snake graph. Moreover, we leave the first and the last half step of the Christoffel path empty, see the upper right picture in  Figure~\ref{linepqch}. 
\begin{figure}
\begin{center}
\scalebox{1}{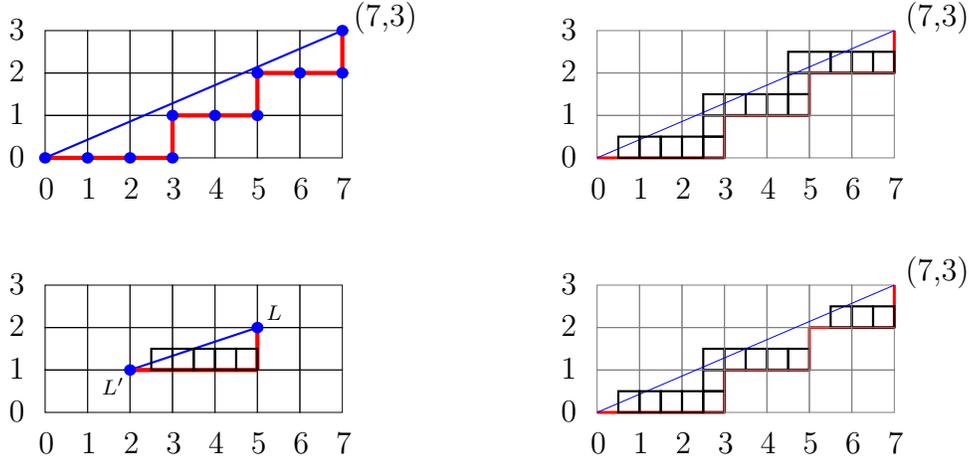}
\caption{The line with slope $p/q=3/7$ with its lower Christoffel path  in red (top left), defining the Christoffel word $xxxyxxyxxy.$ The corresponding snake graph (top right) is obtained by placing tiles of side length 1/2 on the Christoffel path leaving the first half step and the last half step empty.  Its continued fraction is $[2,1,1,2,2,1,1,2,2,1,1,2]=2897/1120$. The bottom right picture shows the two unique snake graphs that complete the Markov triple. Their continued fractions are $[2,1,1,2,2,1,1,2]=194/75 $ and $[2,2]=5/2$. The bottom left picture shows the Markov graph corresponding to the mutation.} 
\label{linepqch}
\end{center}
\end{figure}

Frobenius' uniqueness conjecture is equivalent to the conjecture that no two  Markov snake graphs have the same number of perfect matchings. It is known that every Markov snake graph $\calg$ determines a unique pair of Markov snake graphs $\calg',\calg''$ such that the three graphs form a Markov triple in which $\calg $ is the largest graph. In fact $\calg'$ and $\calg''$ are subgraphs of $\calg$.

The description of the Markov snake graph in terms of the Christoffel path is useful to determine the two smaller Markov snake graphs $\calg'$ and $\calg''$ from $\calg$. The Christoffel path decomposes in a unique way as a concatenation of two Christoffel paths at the lattice point $L$ that is closest to the diagonal, see \cite{BLRS}. In our example, this point is the point $(5,2)$ and the Christoffel word factors as follows $(xxxyxxy)(xxy).$ 
The Markov snake graphs $\calg'$ and $\calg''$ are the graphs of these shorter Christoffel paths.

 We obtain $\calg'$ and $\calg''$ from the original Markov snake graph $\calg$ simply by removing the 3 tiles that are incident to the lattice point $L$, see the bottom right picture in Figure \ref{linepqch}. In that example, the Markov triple is $(2897,194,5)$.

The mutation of the Markov triple $(m_1,m_2,m_3) \to (m_1',m_2,m_3)$ is given by the formula 
$m_1m_1'=m_2^2+m_3^2$. 
The Markov snake graph $\calg_1'$ of $m_1'$ is also easily obtained from our picture. Let $L,L'$ be the lattice points that are closest to the diagonal from below and above, respectively. In our example, we have $L=(5,2)$ and $L'=(7,3)-(5,2)=(2,1)$, see the bottom left picture in Figure \ref{linepqch}. Then the Markov snake graph $\calg_1'$ is the one determined by the line segment between $L'$ and $L$. In our example $\calg'$ is a straight snake graph with 5 tiles. It has 13 perfect matchings, confirming the mutation formula $13= 3\cdot 194\cdot 5-2897= (194^2+5^2)/2897.$

 \subsection{Markov band graphs} Let $m$ be a Markov number and $\calg(m)=\calg_\zg$ its snake graph, where $\zg$ is the  corresponding arc in the torus with one puncture. This arc starts and ends at the puncture. Moving its endpoints infinitesimally away from the puncture but keeping them together, we obtain a closed loop $\zeta$. In other words, $\zeta $ is running parallel to $\zg$ except in a small neighborhood of the puncture, where $\zeta$ goes halfway around the puncture while $\zg$ goes directly to the puncture. There are precisely two ways of doing this, namely passing the puncture on the left or on the right. Both  cases are illustrated in Figure \ref{figzeta}. In both cases, going halfway around the puncture creates three additional crossings with the triangulation. Notice that the two pictures with the curves $\zg$ and $\zeta$ are rotationally symmetric. Therefore both cases are essentially the same. We will also verify this now on the level of band graphs.

The band graph $\band_\zeta$ of $\zeta$ has exactly 3 more tiles than the snake graph $\calg_\zg$ of $\zg$. In Figure \ref{figzeta}, we show the last 3 tiles of $\calg_\zg$ in gray and the 3 new tiles of $\band_\zeta$ in white. The black dots indicate that these vertices (respectively  the edge between them) are identified with the two southern vertices (respectively the edge between them) of the first tile of $\calg_\zg$ to form the band graph $\band_\zeta$.
 
Consider the horizontal segments in $\band_\zeta$.
In the first case, the last horizontal segment of $\calg_\zg$ is extended by two tiles and all other horizontal segments of $\band_\zeta $ are of the same length as the horizontal segments of $\calg_\gamma$. In the second case, it is the first horizontal segment of $\calg_\zg$ that is extended by two tiles and all other horizontal segments of $\band_\zeta $ are of the same length as the horizontal segments of $\calg_\gamma$. However, since $\calg_\gamma$ is a palindromic snake graph, we see that the result in both cases is the same, and thus the band graph $\band_\zeta$ is uniquely determined by $\calg_\zg$.

\begin{figure}
 
\begin{center}
\scalebox{1.2}{ 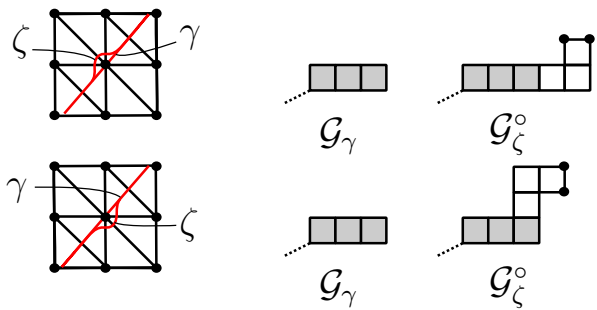}
 \caption{Construction of a Markov band graph. On the left, the band $\zeta$ is obtained from the arc $\zg$ by avoiding the puncture. The pictures on the right show the difference between the snake graph $\calg_\zg$ and the band graph $\band_\zeta$. The band graph has 3 additional  tiles.}\label{figzeta}
\end{center}
\end{figure}

In this way we have associated a band graph $\band_\zeta$ to every Markov number $m$. We shall often use the notation  $\band(m)=\band_\zeta$ in order to emphasize this relation.

\begin{thm}\label{thm band}
 The number of perfect matchings of $\band (m)$ is  $3m$.
\end{thm}
\begin{proof} The proof is a relatively simple computation with snake graphs. In Figure~\ref{example}, we show this computation in the case where $m=5$, omitting single edge snake graphs since they have exactly one perfect matching.  Let $d$ be the number of tiles in the Markov snake graph $\calg_\zg$. 
Denote by $\calg^+$  the snake graph obtained from $\band(m)$ by cutting along the glueing edge. Thus $\calg^+$ has $d+3$ tiles and its initial $d$ tiles form the Markov snake graph $\calg_\zg$. Let $e $ be the first interior edge in $\calg^+$ that has the same sign as the glueing edge, and let $e' $ be the last  interior edge in $\calg^+$ that has the same sign as the glueing edge. 
The self-grafting formula \cite[Section 3.4]{CS2} describes a relation between the band graph $\band(m)$ and its cut $\calg^+$ in the snake ring. It says that $\calg^+ =\band(m) \cdot (\textup{glueing\ edge}) + \calg_-$, where $\calg_-$ is the snake graph obtained from $\calg^+$ by removing the tiles that precede the interior edge $e$ and also  removing the tiles that succeed the interior edge $e'$. Since our snake graph $\calg^+$ is constructed from the Markov snake graph, we know exactly which tiles to remove, namely,  $\calg_-$ is obtained from $\calg^+$ by removing 
the first 2 tiles and the last 3 tiles.  Equivalently, $\calg_-$ is obtained from the Markov snake graph $\calg_\zg$ by removing the first two tiles.

On the other hand, using the formula for grafting with a single edge from \cite[Section 3.3 case 3]{CS2}, we also see that $\calg^+  \cdot (\textup{single\ edge}) = \calg_\zg \calg' +\calg^- \cdot (\textup{single\ edge}) $, where $\calg'$ is the snake graph consisting of the last two tiles of $\calg^+$ and $\calg^-$ is the snake graph obtained from the Markov snake graph $\calg_{\zg}$ by removing the last two tiles. In particular,  the two snake graphs $\calg_-$ and $\calg^-$  are isomorphic, since the Markov snake graph is rotationally symmetric.

Putting these  results together, we see that up to multiplying by single edges we have the following identity in the snake ring 
\[\band(m) =\calg^+ -  \calg_- =  \calg_\zg \calg' +\calg^- -  \calg_- =\calg_\zg \calg'.\]
Now the result follows since the number of perfect matchings of $\calg_\zg$ is $m$ and the number of perfect matchings of $\calg'$ is 3.
\end{proof}

\begin{remark}
 This result has an interesting interpretation in number theory, because the triple $(m_1,m_2,m_3)$ is a solution of the Markov equation $x^2+y^2+z^2=3xyz$ if and only if  the  triple $(3m_1,3m_2,3m_3)$ is  a solution of the equation $x^2+y^2+z^2=xyz$. On the other hand, the Diophantine equation $x^2+y^2+z^2=k\,xyz$ has a positive integer solution if and only if $k=1$ or $k=3$.  Geometrically, it was known that the solution  $(m_1,m_2,m_3)$ corresponds to a triangulation $T$ of the torus with one puncture. And the theorem above shows that the solution $(3m_1,3m_2,3m_3)$  corresponds to the three closed loops obtained by moving the arcs of $T$ infinitesimally away from the puncture.
\end{remark}

\begin{figure}
\begin{center}
\large\scalebox{0.7}{ 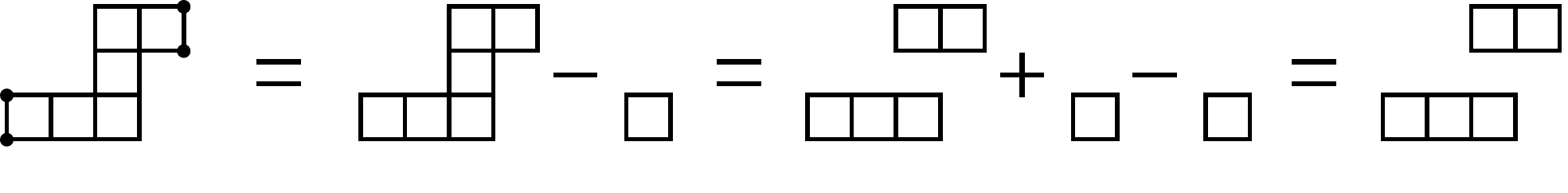}
 \caption{An example illustrating the proof of Theorem \ref{thm band}.}\label{example}
\end{center}
\end{figure}

\end{document}

%% file: FourTiles.pdf_tex
\begingroup%
  \makeatletter%
  \providecommand\color[2][]{%
    \errmessage{(Inkscape) Color is used for the text in Inkscape, but the package 'color.sty' is not loaded}%
    \renewcommand\color[2][]{}%
  }%
  \providecommand\transparent[1]{%
    \errmessage{(Inkscape) Transparency is used (non-zero) for the text in Inkscape, but the package 'transparent.sty' is not loaded}%
    \renewcommand\transparent[1]{}%
  }%
  \providecommand\rotatebox[2]{#2}%
  \ifx\svgwidth\undefined%
    \setlength{\unitlength}{571.79814205bp}%
    \ifx\svgscale\undefined%
      \relax%
    \else%
      \setlength{\unitlength}{\unitlength * \real{\svgscale}}%
    \fi%
  \else%
    \setlength{\unitlength}{\svgwidth}%
  \fi%
  \global\let\svgwidth\undefined%
  \global\let\svgscale\undefined%
  \makeatother%
  \begin{picture}(1,0.14579243)%
    \put(0,0){\includegraphics[width=\unitlength]{FourTiles.pdf}}%
    \put(-0.00093319,0.0033775){\color[rgb]{0,0,0}\makebox(0,0)[lb]{\smash{[2,1,2]=8/3}}}%
    \put(0.16889222,0.0052519){\color[rgb]{0,0,0}\makebox(0,0)[lb]{\smash{[2,2,1]=7/3}}}%
    \put(0.29340982,0.0052519){\color[rgb]{0,0,0}\makebox(0,0)[lb]{\smash{[5]=5/1}}}%
    \put(0.41775948,0.0052519){\color[rgb]{0,0,0}\makebox(0,0)[lb]{\smash{[3,1,1]=7/2}}}%
    \put(0.5305865,0.0052519){\color[rgb]{0,0,0}\makebox(0,0)[lb]{\smash{[1,2,2]=7/5}}}%
    \put(0.67189328,0.0052519){\color[rgb]{0,0,0}\makebox(0,0)[lb]{\smash{[1,3,1]=5/4}}}%
    \put(0.78317578,0.0052519){\color[rgb]{0,0,0}\makebox(0,0)[lb]{\smash{[1,1,3]=7/4}}}%
    \put(0.88706542,0.0052519){\color[rgb]{0,0,0}\makebox(0,0)[lb]{\smash{[1,1,1,1,1]=8/5}}}%
  \end{picture}%
\endgroup%

%% file: snake8437.pdf_tex
\begingroup%
  \makeatletter%
  \providecommand\color[2][]{%
    \errmessage{(Inkscape) Color is used for the text in Inkscape, but the package 'color.sty' is not loaded}%
    \renewcommand\color[2][]{}%
  }%
  \providecommand\transparent[1]{%
    \errmessage{(Inkscape) Transparency is used (non-zero) for the text in Inkscape, but the package 'transparent.sty' is not loaded}%
    \renewcommand\transparent[1]{}%
  }%
  \providecommand\rotatebox[2]{#2}%
  \ifx\svgwidth\undefined%
    \setlength{\unitlength}{298.81789551bp}%
    \ifx\svgscale\undefined%
      \relax%
    \else%
      \setlength{\unitlength}{\unitlength * \real{\svgscale}}%
    \fi%
  \else%
    \setlength{\unitlength}{\svgwidth}%
  \fi%
  \global\let\svgwidth\undefined%
  \global\let\svgscale\undefined%
  \makeatother%
  \begin{picture}(1,0.28529098)%
    \put(0,0){\includegraphics[width=\unitlength]{snake8437.pdf}}%
    \put(0.4969318,0.00765255){\color[rgb]{0,0,0}\makebox(0,0)[lb]{\smash{$\calh_1$}}}%
    \put(0.60402043,0.00765255){\color[rgb]{0,0,0}\makebox(0,0)[lb]{\smash{$\calh_2$}}}%
    \put(0.68433691,0.06119687){\color[rgb]{0,0,0}\makebox(0,0)[lb]{\smash{$\calh_3$}}}%
    \put(0.76465338,0.06119687){\color[rgb]{0,0,0}\makebox(0,0)[lb]{\smash{$\calh_4$}}}%
    \put(0.87174201,0.23789311){\color[rgb]{0,0,0}\makebox(0,0)[lb]{\smash{$\calh_5$}}}%
  \end{picture}%
\endgroup%

%% file: expm.pdf_tex
\begingroup%
  \makeatletter%
  \providecommand\color[2][]{%
    \errmessage{(Inkscape) Color is used for the text in Inkscape, but the package 'color.sty' is not loaded}%
    \renewcommand\color[2][]{}%
  }%
  \providecommand\transparent[1]{%
    \errmessage{(Inkscape) Transparency is used (non-zero) for the text in Inkscape, but the package 'transparent.sty' is not loaded}%
    \renewcommand\transparent[1]{}%
  }%
  \providecommand\rotatebox[2]{#2}%
  \ifx\svgwidth\undefined%
    \setlength{\unitlength}{290.4bp}%
    \ifx\svgscale\undefined%
      \relax%
    \else%
      \setlength{\unitlength}{\unitlength * \real{\svgscale}}%
    \fi%
  \else%
    \setlength{\unitlength}{\svgwidth}%
  \fi%
  \global\let\svgwidth\undefined%
  \global\let\svgscale\undefined%
  \makeatother%
  \begin{picture}(1,0.39393939)%
    \put(0,0){\includegraphics[width=\unitlength]{expm.pdf}}%
  \end{picture}%
\endgroup%

%% file: example8437.pdf_tex
\begingroup%
  \makeatletter%
  \providecommand\color[2][]{%
    \errmessage{(Inkscape) Color is used for the text in Inkscape, but the package 'color.sty' is not loaded}%
    \renewcommand\color[2][]{}%
  }%
  \providecommand\transparent[1]{%
    \errmessage{(Inkscape) Transparency is used (non-zero) for the text in Inkscape, but the package 'transparent.sty' is not loaded}%
    \renewcommand\transparent[1]{}%
  }%
  \providecommand\rotatebox[2]{#2}%
  \ifx\svgwidth\undefined%
    \setlength{\unitlength}{416.8bp}%
    \ifx\svgscale\undefined%
      \relax%
    \else%
      \setlength{\unitlength}{\unitlength * \real{\svgscale}}%
    \fi%
  \else%
    \setlength{\unitlength}{\svgwidth}%
  \fi%
  \global\let\svgwidth\undefined%
  \global\let\svgscale\undefined%
  \makeatother%
  \begin{picture}(1,0.34740883)%
    \put(0,0){\includegraphics[width=\unitlength]{example8437.pdf}}%
    \put(0.0153308,0.20468607){\color[rgb]{0,0,0}\makebox(0,0)[lb]{\smash{2}}}%
    \put(0.05436021,0.20495786){\color[rgb]{0,0,0}\makebox(0,0)[lb]{\smash{3}}}%
    \put(0.09155294,0.20495786){\color[rgb]{0,0,0}\makebox(0,0)[lb]{\smash{5}}}%
    \put(0.09062669,0.24301429){\color[rgb]{0,0,0}\makebox(0,0)[lb]{\smash{7}}}%
    \put(0.12930254,0.24328607){\color[rgb]{0,0,0}\makebox(0,0)[lb]{\smash{9}}}%
    \put(0.16085375,0.24328607){\color[rgb]{0,0,0}\makebox(0,0)[lb]{\smash{16}}}%
    \put(0.20058999,0.24328607){\color[rgb]{0,0,0}\makebox(0,0)[lb]{\smash{25}}}%
    \put(0.20026197,0.28277466){\color[rgb]{0,0,0}\makebox(0,0)[lb]{\smash{34}}}%
    \put(0.20038849,0.31904905){\color[rgb]{0,0,0}\makebox(0,0)[lb]{\smash{59}}}%
    \put(0.23917458,0.31904905){\color[rgb]{0,0,0}\makebox(0,0)[lb]{\smash{84}}}%
    \put(0.09117083,0.01250061){\color[rgb]{0,0,0}\makebox(0,0)[lb]{\smash{2}}}%
    \put(0.09117083,0.05088833){\color[rgb]{0,0,0}\makebox(0,0)[lb]{\smash{3}}}%
    \put(0.12907809,0.05089049){\color[rgb]{0,0,0}\makebox(0,0)[lb]{\smash{4}}}%
    \put(0.16818647,0.05089049){\color[rgb]{0,0,0}\makebox(0,0)[lb]{\smash{7}}}%
    \put(0.20091848,0.05089049){\color[rgb]{0,0,0}\makebox(0,0)[lb]{\smash{11}}}%
    \put(0.20087162,0.09022343){\color[rgb]{0,0,0}\makebox(0,0)[lb]{\smash{15}}}%
    \put(0.20098877,0.12732226){\color[rgb]{0,0,0}\makebox(0,0)[lb]{\smash{26}}}%
    \put(0.23827287,0.12748786){\color[rgb]{0,0,0}\makebox(0,0)[lb]{\smash{37}}}%
    \put(0.36082025,0.20468607){\color[rgb]{0,0,0}\makebox(0,0)[lb]{\smash{2}}}%
    \put(0.39984966,0.20495786){\color[rgb]{0,0,0}\makebox(0,0)[lb]{\smash{3}}}%
    \put(0.43704241,0.20495786){\color[rgb]{0,0,0}\makebox(0,0)[lb]{\smash{5}}}%
    \put(0.43611611,0.24301429){\color[rgb]{0,0,0}\makebox(0,0)[lb]{\smash{7}}}%
    \put(0.47479201,0.24328607){\color[rgb]{0,0,0}\makebox(0,0)[lb]{\smash{9}}}%
    \put(0.5063432,0.24328607){\color[rgb]{0,0,0}\makebox(0,0)[lb]{\smash{16}}}%
    \put(0.54607943,0.24328607){\color[rgb]{0,0,0}\makebox(0,0)[lb]{\smash{25}}}%
    \put(0.43666027,0.01250061){\color[rgb]{0,0,0}\makebox(0,0)[lb]{\smash{2}}}%
    \put(0.43666027,0.05088833){\color[rgb]{0,0,0}\makebox(0,0)[lb]{\smash{3}}}%
    \put(0.47456755,0.05089049){\color[rgb]{0,0,0}\makebox(0,0)[lb]{\smash{4}}}%
    \put(0.51367592,0.05089049){\color[rgb]{0,0,0}\makebox(0,0)[lb]{\smash{7}}}%
    \put(0.54640792,0.05089049){\color[rgb]{0,0,0}\makebox(0,0)[lb]{\smash{11}}}%
    \put(0.66792201,0.20468607){\color[rgb]{0,0,0}\makebox(0,0)[lb]{\smash{2}}}%
    \put(0.70695135,0.20495786){\color[rgb]{0,0,0}\makebox(0,0)[lb]{\smash{3}}}%
    \put(0.74414414,0.20495786){\color[rgb]{0,0,0}\makebox(0,0)[lb]{\smash{5}}}%
    \put(0.74321783,0.24301429){\color[rgb]{0,0,0}\makebox(0,0)[lb]{\smash{7}}}%
    \put(0.78189374,0.24328607){\color[rgb]{0,0,0}\makebox(0,0)[lb]{\smash{9}}}%
    \put(0.743762,0.01250061){\color[rgb]{0,0,0}\makebox(0,0)[lb]{\smash{2}}}%
    \put(0.743762,0.05088833){\color[rgb]{0,0,0}\makebox(0,0)[lb]{\smash{3}}}%
    \put(0.78166928,0.05089049){\color[rgb]{0,0,0}\makebox(0,0)[lb]{\smash{4}}}%
    \put(0.8982483,0.20468607){\color[rgb]{0,0,0}\makebox(0,0)[lb]{\smash{2}}}%
    \put(0.93727765,0.20495786){\color[rgb]{0,0,0}\makebox(0,0)[lb]{\smash{3}}}%
    \put(0.97447044,0.20495786){\color[rgb]{0,0,0}\makebox(0,0)[lb]{\smash{5}}}%
    \put(0.97354413,0.24301429){\color[rgb]{0,0,0}\makebox(0,0)[lb]{\smash{7}}}%
    \put(0.97408829,0.01250061){\color[rgb]{0,0,0}\makebox(0,0)[lb]{\smash{2}}}%
    \put(0.97408829,0.05088833){\color[rgb]{0,0,0}\makebox(0,0)[lb]{\smash{3}}}%
  \end{picture}%
\endgroup%

%% file: example8437Div.pdf_tex
\begingroup%
  \makeatletter%
  \providecommand\color[2][]{%
    \errmessage{(Inkscape) Color is used for the text in Inkscape, but the package 'color.sty' is not loaded}%
    \renewcommand\color[2][]{}%
  }%
  \providecommand\transparent[1]{%
    \errmessage{(Inkscape) Transparency is used (non-zero) for the text in Inkscape, but the package 'transparent.sty' is not loaded}%
    \renewcommand\transparent[1]{}%
  }%
  \providecommand\rotatebox[2]{#2}%
  \ifx\svgwidth\undefined%
    \setlength{\unitlength}{344.92038574bp}%
    \ifx\svgscale\undefined%
      \relax%
    \else%
      \setlength{\unitlength}{\unitlength * \real{\svgscale}}%
    \fi%
  \else%
    \setlength{\unitlength}{\svgwidth}%
  \fi%
  \global\let\svgwidth\undefined%
  \global\let\svgscale\undefined%
  \makeatother%
  \begin{picture}(1,0.82143213)%
    \put(0,0){\includegraphics[width=\unitlength]{example8437Div.pdf}}%
    \put(0.29475715,0.78895209){\color[rgb]{0,0,0}\makebox(0,0)[lb]{\smash{2}}}%
    \put(0.24898706,0.78928052){\color[rgb]{0,0,0}\makebox(0,0)[lb]{\smash{3}}}%
    \put(0.24962126,0.74180813){\color[rgb]{0,0,0}\makebox(0,0)[lb]{\smash{4}}}%
    \put(0.24885116,0.69497168){\color[rgb]{0,0,0}\makebox(0,0)[lb]{\smash{7}}}%
    \put(0.19449089,0.69530014){\color[rgb]{0,0,0}\makebox(0,0)[lb]{\smash{10}}}%
    \put(0.1490199,0.69497168){\color[rgb]{0,0,0}\makebox(0,0)[lb]{\smash{17}}}%
    \put(0.10321222,0.69497168){\color[rgb]{0,0,0}\makebox(0,0)[lb]{\smash{27}}}%
    \put(0.10223023,0.6486651){\color[rgb]{0,0,0}\makebox(0,0)[lb]{\smash{37}}}%
    \put(0.05739521,0.64865378){\color[rgb]{0,0,0}\makebox(0,0)[lb]{\smash{47}}}%
    \put(0.01124149,0.6486651){\color[rgb]{0,0,0}\makebox(0,0)[lb]{\smash{84}}}%
    \put(0.43369394,0.6489667){\color[rgb]{0,0,0}\makebox(0,0)[lb]{\smash{2}}}%
    \put(0.71224479,0.78895209){\color[rgb]{0,0,0}\makebox(0,0)[lb]{\smash{2}}}%
    \put(0.66647471,0.78928052){\color[rgb]{0,0,0}\makebox(0,0)[lb]{\smash{3}}}%
    \put(0.66710891,0.74180813){\color[rgb]{0,0,0}\makebox(0,0)[lb]{\smash{4}}}%
    \put(0.66633881,0.69497168){\color[rgb]{0,0,0}\makebox(0,0)[lb]{\smash{7}}}%
    \put(0.61197858,0.69530014){\color[rgb]{0,0,0}\makebox(0,0)[lb]{\smash{10}}}%
    \put(0.56650754,0.69497168){\color[rgb]{0,0,0}\makebox(0,0)[lb]{\smash{17}}}%
    \put(0.52069987,0.69497168){\color[rgb]{0,0,0}\makebox(0,0)[lb]{\smash{27}}}%
    \put(0.51971784,0.6486651){\color[rgb]{0,0,0}\makebox(0,0)[lb]{\smash{37}}}%
    \put(0.94418238,0.78895209){\color[rgb]{0,0,0}\makebox(0,0)[lb]{\smash{2}}}%
    \put(0.89841229,0.78928052){\color[rgb]{0,0,0}\makebox(0,0)[lb]{\smash{3}}}%
    \put(0.89904649,0.74180813){\color[rgb]{0,0,0}\makebox(0,0)[lb]{\smash{4}}}%
    \put(0.89827639,0.69497168){\color[rgb]{0,0,0}\makebox(0,0)[lb]{\smash{7}}}%
    \put(0.84391616,0.69530014){\color[rgb]{0,0,0}\makebox(0,0)[lb]{\smash{10}}}%
    \put(0.29475715,0.55701451){\color[rgb]{0,0,0}\makebox(0,0)[lb]{\smash{2}}}%
    \put(0.24898706,0.55734294){\color[rgb]{0,0,0}\makebox(0,0)[lb]{\smash{3}}}%
    \put(0.24962126,0.50987055){\color[rgb]{0,0,0}\makebox(0,0)[lb]{\smash{4}}}%
    \put(0.24885116,0.46303406){\color[rgb]{0,0,0}\makebox(0,0)[lb]{\smash{7}}}%
    \put(0.19449093,0.46336256){\color[rgb]{0,0,0}\makebox(0,0)[lb]{\smash{10}}}%
    \put(0.1490199,0.46303406){\color[rgb]{0,0,0}\makebox(0,0)[lb]{\smash{17}}}%
    \put(0.10321222,0.46303406){\color[rgb]{0,0,0}\makebox(0,0)[lb]{\smash{27}}}%
    \put(0.1022302,0.41672752){\color[rgb]{0,0,0}\makebox(0,0)[lb]{\smash{37}}}%
    \put(0.52646901,0.41702912){\color[rgb]{0,0,0}\makebox(0,0)[lb]{\smash{2}}}%
    \put(0.52646901,0.46341664){\color[rgb]{0,0,0}\makebox(0,0)[lb]{\smash{3}}}%
    \put(0.71224479,0.55701451){\color[rgb]{0,0,0}\makebox(0,0)[lb]{\smash{2}}}%
    \put(0.66647471,0.55734294){\color[rgb]{0,0,0}\makebox(0,0)[lb]{\smash{3}}}%
    \put(0.66710891,0.50987055){\color[rgb]{0,0,0}\makebox(0,0)[lb]{\smash{4}}}%
    \put(0.66633881,0.46303406){\color[rgb]{0,0,0}\makebox(0,0)[lb]{\smash{7}}}%
    \put(0.61197858,0.46336256){\color[rgb]{0,0,0}\makebox(0,0)[lb]{\smash{10}}}%
    \put(0.94418238,0.55701451){\color[rgb]{0,0,0}\makebox(0,0)[lb]{\smash{2}}}%
    \put(0.89841229,0.55734294){\color[rgb]{0,0,0}\makebox(0,0)[lb]{\smash{3}}}%
    \put(0.89904649,0.50987055){\color[rgb]{0,0,0}\makebox(0,0)[lb]{\smash{4}}}%
    \put(0.89827639,0.46303406){\color[rgb]{0,0,0}\makebox(0,0)[lb]{\smash{7}}}%
    \put(0.29475715,0.32507693){\color[rgb]{0,0,0}\makebox(0,0)[lb]{\smash{2}}}%
    \put(0.24898706,0.32540535){\color[rgb]{0,0,0}\makebox(0,0)[lb]{\smash{3}}}%
    \put(0.24962126,0.27793297){\color[rgb]{0,0,0}\makebox(0,0)[lb]{\smash{4}}}%
    \put(0.24885116,0.23109648){\color[rgb]{0,0,0}\makebox(0,0)[lb]{\smash{7}}}%
    \put(0.19449093,0.23142498){\color[rgb]{0,0,0}\makebox(0,0)[lb]{\smash{10}}}%
    \put(0.71224479,0.32507693){\color[rgb]{0,0,0}\makebox(0,0)[lb]{\smash{2}}}%
    \put(0.66647471,0.32540535){\color[rgb]{0,0,0}\makebox(0,0)[lb]{\smash{3}}}%
    \put(0.66710891,0.27793297){\color[rgb]{0,0,0}\makebox(0,0)[lb]{\smash{4}}}%
    \put(0.66633881,0.23109648){\color[rgb]{0,0,0}\makebox(0,0)[lb]{\smash{7}}}%
    \put(0.56873653,0.23171688){\color[rgb]{0,0,0}\makebox(0,0)[lb]{\smash{1}}}%
    \put(0.94418238,0.32507693){\color[rgb]{0,0,0}\makebox(0,0)[lb]{\smash{2}}}%
    \put(0.89841229,0.32540535){\color[rgb]{0,0,0}\makebox(0,0)[lb]{\smash{3}}}%
    \put(0.29475715,0.13952686){\color[rgb]{0,0,0}\makebox(0,0)[lb]{\smash{2}}}%
    \put(0.24898706,0.13985529){\color[rgb]{0,0,0}\makebox(0,0)[lb]{\smash{3}}}%
    \put(0.24962126,0.0923829){\color[rgb]{0,0,0}\makebox(0,0)[lb]{\smash{4}}}%
    \put(0.24885116,0.04554642){\color[rgb]{0,0,0}\makebox(0,0)[lb]{\smash{7}}}%
    \put(0.71224479,0.13952686){\color[rgb]{0,0,0}\makebox(0,0)[lb]{\smash{2}}}%
    \put(0.66647471,0.13985529){\color[rgb]{0,0,0}\makebox(0,0)[lb]{\smash{3}}}%
    \put(0.66585728,0.04675183){\color[rgb]{0,0,0}\makebox(0,0)[lb]{\smash{2}}}%
    \put(0.98622418,0.13894185){\color[rgb]{0,0,0}\makebox(0,0)[lb]{\smash{1}}}%
    \put(0.42460447,0.60966741){\color[rgb]{0,0,0}\makebox(0,0)[lb]{\smash{$\calh_1$}}}%
    \put(0.5173795,0.37772983){\color[rgb]{0,0,0}\makebox(0,0)[lb]{\smash{$\calh_2$}}}%
    \put(0.54057326,0.19217977){\color[rgb]{0,0,0}\makebox(0,0)[lb]{\smash{$\calh_3$}}}%
    \put(0.65654205,0.0066297){\color[rgb]{0,0,0}\makebox(0,0)[lb]{\smash{$\calh_4$}}}%
  \end{picture}%
\endgroup%

%% file: example8437DivEven.pdf_tex
\begingroup%
  \makeatletter%
  \providecommand\color[2][]{%
    \errmessage{(Inkscape) Color is used for the text in Inkscape, but the package 'color.sty' is not loaded}%
    \renewcommand\color[2][]{}%
  }%
  \providecommand\transparent[1]{%
    \errmessage{(Inkscape) Transparency is used (non-zero) for the text in Inkscape, but the package 'transparent.sty' is not loaded}%
    \renewcommand\transparent[1]{}%
  }%
  \providecommand\rotatebox[2]{#2}%
  \ifx\svgwidth\undefined%
    \setlength{\unitlength}{344.92038574bp}%
    \ifx\svgscale\undefined%
      \relax%
    \else%
      \setlength{\unitlength}{\unitlength * \real{\svgscale}}%
    \fi%
  \else%
    \setlength{\unitlength}{\svgwidth}%
  \fi%
  \global\let\svgwidth\undefined%
  \global\let\svgscale\undefined%
  \makeatother%
  \begin{picture}(1,0.74451964)%
    \put(0,0){\includegraphics[width=\unitlength]{example8437DivEven.pdf}}%
    \put(0.29475715,0.7120396){\color[rgb]{0,0,0}\makebox(0,0)[lb]{\smash{2}}}%
    \put(0.24898706,0.71236802){\color[rgb]{0,0,0}\makebox(0,0)[lb]{\smash{3}}}%
    \put(0.24962126,0.66489564){\color[rgb]{0,0,0}\makebox(0,0)[lb]{\smash{4}}}%
    \put(0.24885116,0.61805919){\color[rgb]{0,0,0}\makebox(0,0)[lb]{\smash{7}}}%
    \put(0.19449089,0.61838765){\color[rgb]{0,0,0}\makebox(0,0)[lb]{\smash{10}}}%
    \put(0.1490199,0.61805919){\color[rgb]{0,0,0}\makebox(0,0)[lb]{\smash{17}}}%
    \put(0.10321222,0.61805919){\color[rgb]{0,0,0}\makebox(0,0)[lb]{\smash{27}}}%
    \put(0.10223023,0.57175261){\color[rgb]{0,0,0}\makebox(0,0)[lb]{\smash{37}}}%
    \put(0.05739521,0.57174129){\color[rgb]{0,0,0}\makebox(0,0)[lb]{\smash{47}}}%
    \put(0.01124149,0.57175261){\color[rgb]{0,0,0}\makebox(0,0)[lb]{\smash{84}}}%
    \put(0.43369394,0.57205421){\color[rgb]{0,0,0}\makebox(0,0)[lb]{\smash{2}}}%
    \put(0.71224479,0.7120396){\color[rgb]{0,0,0}\makebox(0,0)[lb]{\smash{2}}}%
    \put(0.66647471,0.71236802){\color[rgb]{0,0,0}\makebox(0,0)[lb]{\smash{3}}}%
    \put(0.66710891,0.66489564){\color[rgb]{0,0,0}\makebox(0,0)[lb]{\smash{4}}}%
    \put(0.66633881,0.61805919){\color[rgb]{0,0,0}\makebox(0,0)[lb]{\smash{7}}}%
    \put(0.61197858,0.61838765){\color[rgb]{0,0,0}\makebox(0,0)[lb]{\smash{10}}}%
    \put(0.56650754,0.61805919){\color[rgb]{0,0,0}\makebox(0,0)[lb]{\smash{17}}}%
    \put(0.52069987,0.61805919){\color[rgb]{0,0,0}\makebox(0,0)[lb]{\smash{27}}}%
    \put(0.51971784,0.57175261){\color[rgb]{0,0,0}\makebox(0,0)[lb]{\smash{37}}}%
    \put(0.94418238,0.7120396){\color[rgb]{0,0,0}\makebox(0,0)[lb]{\smash{2}}}%
    \put(0.89841229,0.71236802){\color[rgb]{0,0,0}\makebox(0,0)[lb]{\smash{3}}}%
    \put(0.89904649,0.66489564){\color[rgb]{0,0,0}\makebox(0,0)[lb]{\smash{4}}}%
    \put(0.89827639,0.61805919){\color[rgb]{0,0,0}\makebox(0,0)[lb]{\smash{7}}}%
    \put(0.84391616,0.61838765){\color[rgb]{0,0,0}\makebox(0,0)[lb]{\smash{10}}}%
    \put(0.29475715,0.48010202){\color[rgb]{0,0,0}\makebox(0,0)[lb]{\smash{2}}}%
    \put(0.24898706,0.48043044){\color[rgb]{0,0,0}\makebox(0,0)[lb]{\smash{3}}}%
    \put(0.24962126,0.43295806){\color[rgb]{0,0,0}\makebox(0,0)[lb]{\smash{4}}}%
    \put(0.24885116,0.38612157){\color[rgb]{0,0,0}\makebox(0,0)[lb]{\smash{7}}}%
    \put(0.19449093,0.38645007){\color[rgb]{0,0,0}\makebox(0,0)[lb]{\smash{10}}}%
    \put(0.1490199,0.38612157){\color[rgb]{0,0,0}\makebox(0,0)[lb]{\smash{17}}}%
    \put(0.10321222,0.38612157){\color[rgb]{0,0,0}\makebox(0,0)[lb]{\smash{27}}}%
    \put(0.1022302,0.33981503){\color[rgb]{0,0,0}\makebox(0,0)[lb]{\smash{37}}}%
    \put(0.5171915,0.34011663){\color[rgb]{0,0,0}\makebox(0,0)[lb]{\smash{2}}}%
    \put(0.5171915,0.38650415){\color[rgb]{0,0,0}\makebox(0,0)[lb]{\smash{3}}}%
    \put(0.7215223,0.48010202){\color[rgb]{0,0,0}\makebox(0,0)[lb]{\smash{2}}}%
    \put(0.67575221,0.48043044){\color[rgb]{0,0,0}\makebox(0,0)[lb]{\smash{3}}}%
    \put(0.67638641,0.43295806){\color[rgb]{0,0,0}\makebox(0,0)[lb]{\smash{4}}}%
    \put(0.67561631,0.38612157){\color[rgb]{0,0,0}\makebox(0,0)[lb]{\smash{7}}}%
    \put(0.62125608,0.38645007){\color[rgb]{0,0,0}\makebox(0,0)[lb]{\smash{10}}}%
    \put(0.94418238,0.48010202){\color[rgb]{0,0,0}\makebox(0,0)[lb]{\smash{2}}}%
    \put(0.89841229,0.48043044){\color[rgb]{0,0,0}\makebox(0,0)[lb]{\smash{3}}}%
    \put(0.29475715,0.24816444){\color[rgb]{0,0,0}\makebox(0,0)[lb]{\smash{2}}}%
    \put(0.24898706,0.24849286){\color[rgb]{0,0,0}\makebox(0,0)[lb]{\smash{3}}}%
    \put(0.24962126,0.20102047){\color[rgb]{0,0,0}\makebox(0,0)[lb]{\smash{4}}}%
    \put(0.24885116,0.15418399){\color[rgb]{0,0,0}\makebox(0,0)[lb]{\smash{7}}}%
    \put(0.19449093,0.15451249){\color[rgb]{0,0,0}\makebox(0,0)[lb]{\smash{10}}}%
    \put(0.71224479,0.25744194){\color[rgb]{0,0,0}\makebox(0,0)[lb]{\smash{2}}}%
    \put(0.66647471,0.25777037){\color[rgb]{0,0,0}\makebox(0,0)[lb]{\smash{3}}}%
    \put(0.66710891,0.19174297){\color[rgb]{0,0,0}\makebox(0,0)[lb]{\smash{4}}}%
    \put(0.66633881,0.14490649){\color[rgb]{0,0,0}\makebox(0,0)[lb]{\smash{3}}}%
    \put(0.94418238,0.24816444){\color[rgb]{0,0,0}\makebox(0,0)[lb]{\smash{2}}}%
    \put(0.29475715,0.01622685){\color[rgb]{0,0,0}\makebox(0,0)[lb]{\smash{2}}}%
    \put(0.24898706,0.01655528){\color[rgb]{0,0,0}\makebox(0,0)[lb]{\smash{3}}}%
    \put(0.71224479,0.01622685){\color[rgb]{0,0,0}\makebox(0,0)[lb]{\smash{2}}}%
    \put(0.98622418,0.01564184){\color[rgb]{0,0,0}\makebox(0,0)[lb]{\smash{1}}}%
    \put(0.56505638,0.38657054){\color[rgb]{0,0,0}\makebox(0,0)[lb]{\smash{4}}}%
    \put(0.61995129,0.14490649){\color[rgb]{0,0,0}\makebox(0,0)[lb]{\smash{2}}}%
    \put(0.64730227,0.01622685){\color[rgb]{0,0,0}\makebox(0,0)[lb]{\smash{2}}}%
  \end{picture}%
\endgroup%

%% file: rotation.pdf_tex
\begingroup%
  \makeatletter%
  \providecommand\color[2][]{%
    \errmessage{(Inkscape) Color is used for the text in Inkscape, but the package 'color.sty' is not loaded}%
    \renewcommand\color[2][]{}%
  }%
  \providecommand\transparent[1]{%
    \errmessage{(Inkscape) Transparency is used (non-zero) for the text in Inkscape, but the package 'transparent.sty' is not loaded}%
    \renewcommand\transparent[1]{}%
  }%
  \providecommand\rotatebox[2]{#2}%
  \ifx\svgwidth\undefined%
    \setlength{\unitlength}{417.6bp}%
    \ifx\svgscale\undefined%
      \relax%
    \else%
      \setlength{\unitlength}{\unitlength * \real{\svgscale}}%
    \fi%
  \else%
    \setlength{\unitlength}{\svgwidth}%
  \fi%
  \global\let\svgwidth\undefined%
  \global\let\svgscale\undefined%
  \makeatother%
  \begin{picture}(1,0.12643678)%
    \put(0,0){\includegraphics[width=\unitlength]{rotation.pdf}}%
  \end{picture}%
\endgroup%

%% file: fig5.pdf_tex
\begingroup%
  \makeatletter%
  \providecommand\color[2][]{%
    \errmessage{(Inkscape) Color is used for the text in Inkscape, but the package 'color.sty' is not loaded}%
    \renewcommand\color[2][]{}%
  }%
  \providecommand\transparent[1]{%
    \errmessage{(Inkscape) Transparency is used (non-zero) for the text in Inkscape, but the package 'transparent.sty' is not loaded}%
    \renewcommand\transparent[1]{}%
  }%
  \providecommand\rotatebox[2]{#2}%
  \ifx\svgwidth\undefined%
    \setlength{\unitlength}{81.603112bp}%
    \ifx\svgscale\undefined%
      \relax%
    \else%
      \setlength{\unitlength}{\unitlength * \real{\svgscale}}%
    \fi%
  \else%
    \setlength{\unitlength}{\svgwidth}%
  \fi%
  \global\let\svgwidth\undefined%
  \global\let\svgscale\undefined%
  \makeatother%
  \begin{picture}(1,0.60781996)%
    \put(0,0){\includegraphics[width=\unitlength]{fig5.pdf}}%
  \end{picture}%
\endgroup%

%% file: linepqch.pdf_tex
\begingroup%
  \makeatletter%
  \providecommand\color[2][]{%
    \errmessage{(Inkscape) Color is used for the text in Inkscape, but the package 'color.sty' is not loaded}%
    \renewcommand\color[2][]{}%
  }%
  \providecommand\transparent[1]{%
    \errmessage{(Inkscape) Transparency is used (non-zero) for the text in Inkscape, but the package 'transparent.sty' is not loaded}%
    \renewcommand\transparent[1]{}%
  }%
  \providecommand\rotatebox[2]{#2}%
  \ifx\svgwidth\undefined%
    \setlength{\unitlength}{357.70961914bp}%
    \ifx\svgscale\undefined%
      \relax%
    \else%
      \setlength{\unitlength}{\unitlength * \real{\svgscale}}%
    \fi%
  \else%
    \setlength{\unitlength}{\svgwidth}%
  \fi%
  \global\let\svgwidth\undefined%
  \global\let\svgscale\undefined%
  \makeatother%
  \begin{picture}(1,0.46956944)%
    \put(0,0){\includegraphics[width=\unitlength]{linepqch.pdf}}%
    \put(0.03093247,0.26983086){\color[rgb]{0,0,0}\makebox(0,0)[lb]{\smash{0}}}%
    \put(0.36244332,0.45145942){\color[rgb]{0,0,0}\makebox(0,0)[lb]{\smash{(7,3)}}}%
    \put(0.07566149,0.26983086){\color[rgb]{0,0,0}\makebox(0,0)[lb]{\smash{1}}}%
    \put(0.12039051,0.26983086){\color[rgb]{0,0,0}\makebox(0,0)[lb]{\smash{2}}}%
    \put(0.16511953,0.26983086){\color[rgb]{0,0,0}\makebox(0,0)[lb]{\smash{3}}}%
    \put(0.20984855,0.26983086){\color[rgb]{0,0,0}\makebox(0,0)[lb]{\smash{4}}}%
    \put(0.25457756,0.26983086){\color[rgb]{0,0,0}\makebox(0,0)[lb]{\smash{5}}}%
    \put(0.29930658,0.26983086){\color[rgb]{0,0,0}\makebox(0,0)[lb]{\smash{6}}}%
    \put(0.3440356,0.26983086){\color[rgb]{0,0,0}\makebox(0,0)[lb]{\smash{7}}}%
    \put(-0.00037784,0.30337762){\color[rgb]{0,0,0}\makebox(0,0)[lb]{\smash{0}}}%
    \put(-0.00037784,0.34810664){\color[rgb]{0,0,0}\makebox(0,0)[lb]{\smash{1}}}%
    \put(-0.00037784,0.39283566){\color[rgb]{0,0,0}\makebox(0,0)[lb]{\smash{2}}}%
    \put(-0.00037784,0.43756468){\color[rgb]{0,0,0}\makebox(0,0)[lb]{\smash{3}}}%
    \put(0.61240971,0.26983086){\color[rgb]{0,0,0}\makebox(0,0)[lb]{\smash{0}}}%
    \put(0.94392056,0.45145943){\color[rgb]{0,0,0}\makebox(0,0)[lb]{\smash{(7,3)}}}%
    \put(0.65713873,0.26983086){\color[rgb]{0,0,0}\makebox(0,0)[lb]{\smash{1}}}%
    \put(0.70186774,0.26983086){\color[rgb]{0,0,0}\makebox(0,0)[lb]{\smash{2}}}%
    \put(0.74659676,0.26983086){\color[rgb]{0,0,0}\makebox(0,0)[lb]{\smash{3}}}%
    \put(0.79132578,0.26983086){\color[rgb]{0,0,0}\makebox(0,0)[lb]{\smash{4}}}%
    \put(0.8360548,0.26983086){\color[rgb]{0,0,0}\makebox(0,0)[lb]{\smash{5}}}%
    \put(0.88078382,0.26983086){\color[rgb]{0,0,0}\makebox(0,0)[lb]{\smash{6}}}%
    \put(0.92551284,0.26983086){\color[rgb]{0,0,0}\makebox(0,0)[lb]{\smash{7}}}%
    \put(0.5810994,0.30337762){\color[rgb]{0,0,0}\makebox(0,0)[lb]{\smash{0}}}%
    \put(0.5810994,0.34810664){\color[rgb]{0,0,0}\makebox(0,0)[lb]{\smash{1}}}%
    \put(0.5810994,0.39283566){\color[rgb]{0,0,0}\makebox(0,0)[lb]{\smash{2}}}%
    \put(0.5810994,0.43756468){\color[rgb]{0,0,0}\makebox(0,0)[lb]{\smash{3}}}%
    \put(0.61240971,0.00145675){\color[rgb]{0,0,0}\makebox(0,0)[lb]{\smash{0}}}%
    \put(0.94392056,0.18308532){\color[rgb]{0,0,0}\makebox(0,0)[lb]{\smash{(7,3)}}}%
    \put(0.65713873,0.00145675){\color[rgb]{0,0,0}\makebox(0,0)[lb]{\smash{1}}}%
    \put(0.70186774,0.00145675){\color[rgb]{0,0,0}\makebox(0,0)[lb]{\smash{2}}}%
    \put(0.74659676,0.00145675){\color[rgb]{0,0,0}\makebox(0,0)[lb]{\smash{3}}}%
    \put(0.79132578,0.00145675){\color[rgb]{0,0,0}\makebox(0,0)[lb]{\smash{4}}}%
    \put(0.8360548,0.00145675){\color[rgb]{0,0,0}\makebox(0,0)[lb]{\smash{5}}}%
    \put(0.88078382,0.00145675){\color[rgb]{0,0,0}\makebox(0,0)[lb]{\smash{6}}}%
    \put(0.92551284,0.00145675){\color[rgb]{0,0,0}\makebox(0,0)[lb]{\smash{7}}}%
    \put(0.5810994,0.03500351){\color[rgb]{0,0,0}\makebox(0,0)[lb]{\smash{0}}}%
    \put(0.5810994,0.07973253){\color[rgb]{0,0,0}\makebox(0,0)[lb]{\smash{1}}}%
    \put(0.5810994,0.12446155){\color[rgb]{0,0,0}\makebox(0,0)[lb]{\smash{2}}}%
    \put(0.5810994,0.16919057){\color[rgb]{0,0,0}\makebox(0,0)[lb]{\smash{3}}}%
    \put(0.03093247,0.00145675){\color[rgb]{0,0,0}\makebox(0,0)[lb]{\smash{0}}}%
    \put(0.07566149,0.00145675){\color[rgb]{0,0,0}\makebox(0,0)[lb]{\smash{1}}}%
    \put(0.12039051,0.00145675){\color[rgb]{0,0,0}\makebox(0,0)[lb]{\smash{2}}}%
    \put(0.16511953,0.00145675){\color[rgb]{0,0,0}\makebox(0,0)[lb]{\smash{3}}}%
    \put(0.20984855,0.00145675){\color[rgb]{0,0,0}\makebox(0,0)[lb]{\smash{4}}}%
    \put(0.25457756,0.00145675){\color[rgb]{0,0,0}\makebox(0,0)[lb]{\smash{5}}}%
    \put(0.29930658,0.00145675){\color[rgb]{0,0,0}\makebox(0,0)[lb]{\smash{6}}}%
    \put(0.3440356,0.00145675){\color[rgb]{0,0,0}\makebox(0,0)[lb]{\smash{7}}}%
    \put(-0.00037784,0.03500351){\color[rgb]{0,0,0}\makebox(0,0)[lb]{\smash{0}}}%
    \put(-0.00037784,0.07973253){\color[rgb]{0,0,0}\makebox(0,0)[lb]{\smash{1}}}%
    \put(-0.00037784,0.12446155){\color[rgb]{0,0,0}\makebox(0,0)[lb]{\smash{2}}}%
    \put(-0.00037784,0.16919057){\color[rgb]{0,0,0}\makebox(0,0)[lb]{\smash{3}}}%
    \put(0.27090397,0.14121719){\color[rgb]{0,0,0}\makebox(0,0)[lb]{\smash{$\scriptstyle L$}}}%
    \put(0.09753869,0.06318991){\color[rgb]{0,0,0}\makebox(0,0)[lb]{\smash{$\scriptstyle L'$}}}%
  \end{picture}%
\endgroup%

%% file: zeta.pdf_tex
\begingroup%
  \makeatletter%
  \providecommand\color[2][]{%
    \errmessage{(Inkscape) Color is used for the text in Inkscape, but the package 'color.sty' is not loaded}%
    \renewcommand\color[2][]{}%
  }%
  \providecommand\transparent[1]{%
    \errmessage{(Inkscape) Transparency is used (non-zero) for the text in Inkscape, but the package 'transparent.sty' is not loaded}%
    \renewcommand\transparent[1]{}%
  }%
  \providecommand\rotatebox[2]{#2}%
  \ifx\svgwidth\undefined%
    \setlength{\unitlength}{208.01813965bp}%
    \ifx\svgscale\undefined%
      \relax%
    \else%
      \setlength{\unitlength}{\unitlength * \real{\svgscale}}%
    \fi%
  \else%
    \setlength{\unitlength}{\svgwidth}%
  \fi%
  \global\let\svgwidth\undefined%
  \global\let\svgscale\undefined%
  \makeatother%
  \begin{picture}(1,0.44566078)%
    \put(0,0){\includegraphics[width=\unitlength]{zeta.pdf}}%
    \put(0.26024623,0.39495838){\color[rgb]{0,0,0}\makebox(0,0)[lb]{\smash{$\zg$}}}%
    \put(0.00653635,0.38238248){\color[rgb]{0,0,0}\makebox(0,0)[lb]{\smash{$\zeta$}}}%
    \put(0.47100167,0.24174193){\color[rgb]{0,0,0}\makebox(0,0)[lb]{\smash{$\calg_\zg$}}}%
    \put(0.72772432,0.01619661){\color[rgb]{0,0,0}\makebox(0,0)[lb]{\smash{$\band_\zeta$}}}%
    \put(0.26036037,0.10548355){\color[rgb]{0,0,0}\makebox(0,0)[lb]{\smash{$\zeta$}}}%
    \put(-0.00126942,0.16420928){\color[rgb]{0,0,0}\makebox(0,0)[lb]{\smash{$\zg$}}}%
    \put(0.47100167,0.01099288){\color[rgb]{0,0,0}\makebox(0,0)[lb]{\smash{$\calg_\zg$}}}%
    \put(0.72772432,0.24694572){\color[rgb]{0,0,0}\makebox(0,0)[lb]{\smash{$\band_\zeta$}}}%
  \end{picture}%
\endgroup%

%% file: example.pdf_tex
\begingroup%
  \makeatletter%
  \providecommand\color[2][]{%
    \errmessage{(Inkscape) Color is used for the text in Inkscape, but the package 'color.sty' is not loaded}%
    \renewcommand\color[2][]{}%
  }%
  \providecommand\transparent[1]{%
    \errmessage{(Inkscape) Transparency is used (non-zero) for the text in Inkscape, but the package 'transparent.sty' is not loaded}%
    \renewcommand\transparent[1]{}%
  }%
  \providecommand\rotatebox[2]{#2}%
  \ifx\svgwidth\undefined%
    \setlength{\unitlength}{567.04158438bp}%
    \ifx\svgscale\undefined%
      \relax%
    \else%
      \setlength{\unitlength}{\unitlength * \real{\svgscale}}%
    \fi%
  \else%
    \setlength{\unitlength}{\svgwidth}%
  \fi%
  \global\let\svgwidth\undefined%
  \global\let\svgscale\undefined%
  \makeatother%
  \begin{picture}(1,0.1225443)%
    \put(0,0){\includegraphics[width=\unitlength]{example.pdf}}%
    \put(0.02070541,0.0035815){\color[rgb]{0,0,0}\makebox(0,0)[lb]{\smash{$\band(5)$}}}%
    \put(0.26099346,0.00361939){\color[rgb]{0,0,0}\makebox(0,0)[lb]{\smash{$\calg^+$}}}%
    \put(0.40032033,0.00361939){\color[rgb]{0,0,0}\makebox(0,0)[lb]{\smash{$\calg_-$}}}%
    \put(0.54901346,0.00361939){\color[rgb]{0,0,0}\makebox(0,0)[lb]{\smash{$\calg_\zg$}}}%
    \put(0.54208476,0.09631446){\color[rgb]{0,0,0}\makebox(0,0)[lb]{\smash{$\calg'$}}}%
    \put(0.77277979,0.00361939){\color[rgb]{0,0,0}\makebox(0,0)[lb]{\smash{$\calg_-$}}}%
    \put(0.68530825,0.00361939){\color[rgb]{0,0,0}\makebox(0,0)[lb]{\smash{$\calg^-$}}}%
    \put(0.91582959,0.00361939){\color[rgb]{0,0,0}\makebox(0,0)[lb]{\smash{$\calg_\zg$}}}%
    \put(0.90890089,0.09631446){\color[rgb]{0,0,0}\makebox(0,0)[lb]{\smash{$\calg'$}}}%
  \end{picture}%
\endgroup%